\documentclass[10pt, a4paper]{amsart} 
\usepackage{color}
\usepackage{amscd,amssymb,graphics}
\usepackage{amsfonts}
\usepackage{amsmath}
\usepackage{tikz-cd}
\usepackage{multirow}
\usepackage{xparse}
\usepackage[createShortEnv, conf={no link to theorem}]{proof-at-the-end}
\usepackage{faktor}
\usepackage{makecell}
\input xy
\xyoption{all}
\usepackage{epsfig}

\newtheorem{thm}{Theorem}[section]

\newtheorem{f}[thm]{Fact}
\newtheorem{cor}[thm]{Corollary}

\newtheorem{lemma}[thm]{Lemma}
\newtheorem{prop}[thm]{Proposition}

\newtheorem{q}[thm]{Question}

\theoremstyle{definition}
\newtheorem{defin}[thm]{Definition}

\theoremstyle{remark}
\newtheorem{remark}[thm]{Remark}

\newtheorem{ex}[thm]{Example}

\numberwithin{equation}{section}

\newcommand{\delete}[1]{} 

\def\eps{{\varepsilon}}
\newcommand{\sk}{\vskip 0.1cm}

\newcommand{\ben}{\begin{enumerate}}
	
	\newcommand{\een}{\end{enumerate}}
\newcommand{\bit}{\begin{itemize}}
	
	\newcommand{\eit}{\end{itemize}}

\def\A {{\mathcal A}}
\def\R {{\mathbb R}}
\def\N {{\mathbb N}}
\def\Z {{\mathbb Z}}

\def\Or {{\mathcal O}}
\def\Qr {{\mathcal Q}}

\def\arI {{\,\boxdot\,}}
\def\arII {{\,\diamond\,}}

\newcommand{\RUC}{\mathrm{RUC}_{b}}

\newcommand{\ba}{\mathrm{ba}}

\def\Iso{{\mathrm{Iso}}\,}

\def\diam{{\mathrm{diam\ }}}

\def\A{{\mathcal{A}}}

\def\QED{\nobreak\quad\ifmmode\roman{Q.E.D.}\else{\rm Q.E.D.}\fi}

\def\spn{\operatorname{Span}}

\def\tame{\operatorname{\mathit{tame}}}
\def\wap{\operatorname{\mathit{wap}}}
\def\asp{\operatorname{\mathit{Asp}}}

\newcommand{\co}{{\rm{co\,}}}
\newcommand{\ext}{{\rm{ext\,}}}
\newcommand{\acx}{{\rm{acx\,}}}

\graphicspath{ {media/} }

\usepackage[hidelinks]{hyperref}
\begin{document}
	
	\numberwithin{equation}{thm}
	\title[]
	{Characterizations of Asplund and Tame Functionals using Arens Products} 
	
	\sk

	\author[]{Matan Komisarchik}
	\address{Department of Mathematics,
		Bar-Ilan University, 52900 Ramat-Gan, Israel}
	\email{matan.komisarchik@biu.ac.il}
	
	
	\date{February 2026}

	\begin{abstract}	
		We investigate the interaction between Arens products on the bidual of a Banach algebra and structural regularity properties of functionals on the algebra. Building on the classical characterization of weakly almost periodic functionals via Arens regularity, we prove new analogous criteria for Asplund and tame functionals.
		
		We establish a systematic correspondence between geometric properties of orbit sets in the dual---namely weak compactness, fragmentability, and absence of $\ell^1$-sequences---and structural properties of the corresponding bidual orbits under the Arens products, such as weak compactness, separability, and co-tameness. In particular, we obtain bidual characterizations of right Asplund and right tame functionals analogous to the classical weakly almost periodic theory.
		
		We then apply the theory to the group algebra $L^{1}(G)$ of a locally compact group $G$. In this setting, we derive concrete characterizations of Asplund and tame elements of $L^{\infty}(G)$ using orbits of finitely additive $\{0, 1\}$-valued measures.
		For characteristic functions over countable discrete groups, this yields a simple criterion based on countability of the orbit, generalizing a result of Glasner and Megrelishvili for $\ell^1(\mathbb{Z})$.
	\end{abstract}


	

	\maketitle
	\setcounter{tocdepth}{1}
	\tableofcontents
	
	\section{Introduction}
	
	Let $\A$ be a Banach algebra, with dual $\A^{*}$ and bidual $\A^{**}$. The canonical $\A$-bimodule structure on $\A^{*}$ extends to an $\A^{**}$-action, which we denote by $\circledast$. This extension gives rise to two associative multiplications on $\A^{**}$, the \emph{first} and \emph{second Arens products} \cite{Arens}, denoted $\arI$ and $\arII$.
	
	A classical and influential line of research concerns \emph{weakly almost periodic} (WAP) functionals. These are the elements $\varphi \in \A^{*}$ whose orbits $B_{\A} \circledast \varphi$ are relatively weakly compact. Equivalently, they admit characterizations in terms of continuity properties of the orbit maps, their algebraic relationship with the two Arens products, and the orbit structure induced by the bidual; see Proposition~\ref{prop:wap_equivalence} and \cite{Ulger,palmer,pymConvFunc}.
	
	Following a suggestion of M.~Megrelishvili, and taking inspiration from J.~Pym \cite{Pym1993}, the present paper extends this Arens-theoretic viewpoint beyond the WAP setting, focusing instead on the broader and increasingly important classes of \emph{Asplund} and \emph{tame} functionals \cite{TameFunc}. These classes arise naturally in Banach space theory and topological dynamics, where they correspond to low complexity systems. Specifically, they are related to fragmentability and absence of $\ell^{1}$-sequences \cite{GM-rose, GM-MTame}. 
	We believe that their usefulness will carry over to the realm of Banach algebras as well.
	
	A central goal of this paper is to provide a unified ``dictionary'' translating orbit regularity between $\A$ and its bidual: weak compactness (WAP), separability/fragmentability (Asplund), and co-tameness (tameness) emerge as three facets of the same Arens-theoretic picture. The full parallelism is laid out in the table of Theorem~\ref{thm:summary}. In particular, we relate:
	\begin{itemize}
		\item geometric properties of the right orbit $\varphi \circledast B_{\A}$ (weak compactness / Asplundness / tameness),
		\item properties of the left orbit map $\Or_{\varphi}^{(l)}\colon \A^{**} \to \A^{*}$ (continuity / fragmentability),
		\item properties of families of Arens multiplication operators on $\A^{**}$, and
		\item weak compactness/separability, and co-tameness of the bidual orbit $B_{\A^{**}} \circledast \varphi$.
	\end{itemize}
	
	We present two applications. First, in Section~\ref{section:directions_and_examples} we illustrate the abstract criteria on algebras of bounded operators, and then on operator algebras arising from group representations. In both settings we obtain concrete sufficient conditions (Propositions~\ref{prop:sufficient_conditions_for_matrix_coefficients} and~\ref{prop:sufficient_conditions_for_group_matrix_coefficients}), and we view these results as a starting point for further work.
	
	As our main application, we specialize to the group algebra $\A=L^{1}(G)$ of a locally compact group $G$. 
	Identifying $\A^{*}$ with $L^{\infty}(G)$, we obtain concrete descriptions of right Asplund and right tame elements of $L^{\infty}(G)$ in terms of orbits under finitely additive $\{0,1\}$-valued measures (Theorems~\ref{thm:asplund_subsets_on_locally_compact_groups} and~\ref{thm:right_tame_left_co_tame_locally_compact}). 
	Applied to countable, discrete groups, this can be seen as a generalization of a result of Glasner and Megrelishvili \cite[Thm.~4.11]{GM-MTame} about the dynamics of characteristic functions in $\Z$.
	
	The $L^1(G)$ case is especially significant in light of our recent work \cite{meTameFunc}, where it was shown, answering a question posed in \cite{TameFunc}, that a right uniformly continuous function on $G$ is Asplund (respectively tame) as a function if and only if it is Asplund (respectively tame) as a functional on $L^{1}(G)$. 
	Together, these results show that dynamical regularity properties of functions on $G$ are reflected in the Arens-product orbit structure of $L^{1}(G)^{**}$, and in particular in the action of finitely additive $\{0,1\}$-valued measures.
	
	The paper is organized as follows. Section~\ref{section:preliminaries} collects background on Arens products, orbit maps, and fragmentation. We then review classical characterizations of weakly almost periodic functionals and develop the Asplund and tame analogues. The final part of the paper is devoted to applications, with particular emphasis on the group algebra $L^{1}(G)$.
	
	\section{Preliminaries} \label{section:preliminaries}
	All vector spaces in this paper are considered over the field $\R$ of real numbers.
	Also, we assume that all topological spaces are Tychonoff.
	If $X$ is a topological space, let $C(X)$ and $C_{b}(X)$ be the spaces of continuous, and bounded continuous functions over $X$, respectively. 
	A subset $A$ of a Banach space $V$ is weakly relatively compact if its weak closure is weakly compact.
	Write $B_{V}$ for the unit ball of $V$.
	
	Suppose that $V$ is a Banach space. 
	We will write $J\colon V \to V^{**}$ for the natural evaluation embedding and generally also treat $V$ as a subspace of $V^{**}$.
	We will use bracket notation $\langle \cdot, \cdot \rangle$ for duality of Banach spaces. 
	We will write $(V^{*}, w^{*})$ to refer to the weak-star topology of $V^{*}$.
	If $x \in V$, we will write $\rho_{x}$ for the seminorm on $V^{*}$ defined via evaluations at $x$.
	Let $(M, d)$ be a metric space. 
	If $A \subseteq M$ then:
	$$
	  \diam_{d}(A) := \sup\limits_{a, b \in A} d(a, b).
	$$
	
	\begin{f}[Goldstine's Theorem]\cite[p.~IV.17, Prop.~5]{bourbTVS}\label{fact:goldstines}
		If $V$ is a Banach space, then $J(B_{V})$ is weak-star dense in the bidual ball $B_{V^{**}}$.
	\end{f}
	
	\begin{f}[Mazur's Theorem] \cite[p.~65, Cor.~2]{Schaefer} \label{fact:mazur}   
		Let $C$ be a convex subset of a locally convex space $E$.
		Then the closure of $C$ is identical with its weak closure. Hence, $\overline{co (S)}=\overline{co (S)}^w$ 
		for every $S \subset E$.   
	\end{f}
	
	\begin{f}\cite[Thm.~8.10.5, 8.10.13]{Narici2010} \label{fact:adjoint_is_weak_and_weak_star_cont}
		Let $T\colon V \to W$ be a bounded linear map between Banach spaces.
		Then $T$ is (weak, weak) continuous and $T^{*}$ is both (weak-star, weak-star) and (weak, weak) continuous.
	\end{f}
	%
	
	\subsection{Banach Algebras and Arens Products}
	Our main focus in this text will be Banach algebras and Banach modules.
	For their definitions and detailed study, the reader is referred to \cite[Def~1.1.7 and~4.1.17]{palmer}.
	Throughout this text, let $\A$ be a Banach algebra.
	We will use implicit product notation for elements of the algebra.
	As a convention, we will use lowercase letters like $a, b, c$ to represent elements of $\A$ itself, and Greek letters like $\varphi$ and $\psi$ to represent functionals of $\A^{*}$.
	Finally, elements of the bidual $\A^{**}$ will be denoted by Greek letters like $\mu, \nu, \kappa$.
	
	We will now go over the definitions of the first and second Arens products and some of their properties.
	As in \cite{Arens}, the dual $\A^{*}$ can be furnished with a normed $\A$-bimodule structure via:
	\begin{align*}
		\forall a, b \in \A, \varphi \in \A^{*}: \langle a \circledast \varphi, b\rangle & := \langle \varphi, ba \rangle\\
		\forall a, b \in \A, \varphi \in \A^{*}: \langle \varphi \circledast a, b\rangle & := \langle \varphi, ab\rangle.
	\end{align*}
	
	\begin{f} \cite[Thm.~1.4.2]{palmer} \label{lemma:algebra_acts_on_dual}
		For every $\varphi \in \A^{*}, a, b \in \A$ we have:
		\begin{align*}
			(a b) \circledast \varphi & = a \circledast (b \circledast \varphi),\\
			(\varphi \circledast a) \circledast b & = \varphi \circledast (ab),\\
			(a \circledast \varphi) \circledast b & = a \circledast (\varphi \circledast b).
		\end{align*}
	\end{f}
	
	\begin{f}\cite[Thm.~1.4.2]{palmer} \label{lemma:algebra_action_on_dual_is_norm_continuous}
		For every $a \in \A, \varphi \in \A^{*}$ we have:
		$$
		\lVert a \circledast \varphi \rVert, \lVert \varphi \circledast a \rVert \leq \lVert a \rVert \lVert \varphi \rVert.
		$$
		In particular, the left and right actions of $\A$ on $\A^{*}$ are norm-continuous.
	\end{f}

	The action of $\A$ on $\A^{*}$ can be extended to $\A^{**}$ in the following way.
	\begin{defin} \cite[Def.~1.4.1]{palmer}
		For every $\mu \in \A^{**}, \varphi \in \A^{*}$ and $b \in \A$:
		\begin{align*}
			\langle \mu \circledast \varphi, b\rangle & := \langle \mu, \varphi \circledast b \rangle\\
			\langle \varphi \circledast \mu, b\rangle & := \langle \mu, b \circledast \varphi \rangle.
		\end{align*}
	\end{defin}
	
	\begin{f}\cite[Thm.~1.4.2]{palmer} \label{fact:bidual_action_properties}
		For every $a \in \A$, $\varphi \in \A^{*}$ and $\mu \in \A^{**}$:
		\begin{align*}
			a \circledast \varphi & = J(a) \circledast \varphi
			\\
			\varphi \circledast a & = \varphi \circledast J(a)
			\\
			\lVert \mu \circledast \varphi \rVert & \leq \lVert \mu \rVert \lVert \varphi \rVert\\
			\lVert \varphi \circledast \mu\rVert & \leq \lVert \mu \rVert \lVert \varphi \rVert\\
			(\mu \circledast \varphi) \circledast a & =
			\mu \circledast (\varphi \circledast a)\\
			(a \circledast \varphi) \circledast \mu & =
			a \circledast (\varphi \circledast \mu).
		\end{align*}
	\end{f}
	\begin{defin}\cite[Def.~1.4.1]{palmer}
		We can define the two Arens products $\arI, \arII \colon \A^{**} \times \A^{**} \to \A^{**}$ for every $\mu, \nu \in \A^{**}$ and $\varphi \in \A^{*}$:
		\begin{align*}
			\langle \mu \arI \nu, \varphi \rangle & := 
			\langle \mu, \nu \circledast \varphi \rangle,\\
			\langle \mu \arII\nu, \varphi \rangle & :=
			\langle \nu, \varphi \circledast \mu\rangle.
		\end{align*}
	\end{defin}
	
	\begin{f}\cite[Thm.~1.4.2]{palmer} \label{lemma:arens_properties}
		For every $\mu, \nu, \kappa \in \A^{**}, \varphi, \psi \in \A^{*}, a, b \in \A$ we have:
		\begin{align*}
			(\mu \arI \nu) \circledast \varphi & = \mu \circledast (\nu \circledast \varphi)
			\\
			\varphi \circledast (\mu \arII \nu) & = (\varphi \circledast \mu) \circledast \nu
			\\
			(\mu \arI \nu) \arI \kappa & = \mu \arI (\nu \arI \kappa)\\
			(\mu \arII \nu) \arII \kappa & = \mu \arII (\nu \arII \kappa)\\
			J(a) \arI \mu & = J(a) \arII \mu\\
			\mu \arI J(a) & = \mu \arII J(a)\\
			J(a) \arI J(b) = J(a) & \arII J(b) = J(ab)\\
			\lVert \mu \arI \nu \rVert & \leq \lVert \mu \rVert \lVert \nu \rVert\\
			\lVert \mu \arII \nu \rVert & \leq \lVert \mu \rVert \lVert \nu \rVert.
		\end{align*}
	\end{f}
	\begin{remark}
		We will henceforth use implicit product notation for elements of $\A$ and $\A^{**}$, namely: $a\mu$ and $\mu a$.
	\end{remark}
	
	\subsection{Weak-Star Continuity}
	\begin{defin}
		For every $\varphi \in \A^{*}$ we define the \emph{left and right orbit maps $\Or_{\varphi}^{(l)}$ and $\Or_{\varphi}^{(r)}$} via:
		$$
		\Or_{\varphi}^{(l)}(\mu) := \mu \circledast \varphi \text{ and }
		\Or_{\varphi}^{(r)}(\mu) := \varphi \circledast \mu.
		$$
		We will write $\Qr_{\varphi}^{(l)}$ and $\Qr_{\varphi}^{(r)}$ for the appropriate restrictions to $\A$.
	\end{defin}
	The terminology of orbits is justified because it agrees with the orbits of the multiplicative semigroup of $\A^{**}$.
	
	\begin{f}\cite[p.~60]{palmer}\label{lemma:adjoint_representation}
		$\left(\Qr_{\varphi}^{(l)} \right)^{*} = \Or_{\varphi}^{(r)}$ and $\left(\Qr_{\varphi}^{(r)} \right)^{*} = \Or_{\varphi}^{(l)}$.
	\end{f}
	
	\begin{lemma} \label{lemma:bidual_action_is_orbit_weak_star_continuous}
		The orbit maps $\Or_{\varphi}^{(l)}$ and $\Or_{\varphi}^{(r)}$ of $\A^{**}$ on $\A^{*}$ are (weak-star, weak-star) continuous.
		Specifically, if $b \in \A, \varphi \in \A^{*}$ and $\{\mu_{\lambda} \}_{\lambda \in \Lambda}$ is a net in $\A^{**}$ whose limit is $\mu \in \A^{**}$ then:
		\begin{align*}
			\lim\limits_{\lambda \in \Lambda} \langle \mu_{\lambda} \circledast \varphi, b\rangle & = 
			\langle \mu \circledast \varphi, b \rangle\\
			\lim\limits_{\lambda \in \Lambda} \langle \varphi \circledast \mu_{\lambda}, b\rangle & = 
			\langle \varphi \circledast \mu, b \rangle.
		\end{align*}
	\end{lemma}
	\begin{proof}
		A consequence of Fact \ref{lemma:adjoint_representation} and Fact \ref{fact:adjoint_is_weak_and_weak_star_cont}.
	\end{proof}

	\begin{cor} \label{cor:dual_ball_action_is_weak_star_closure}
		For every functional $\varphi \in \A^{*}$ we have:
		$$
		B_{\A^{**}} \circledast \varphi = \overline{B_{\A} \circledast \varphi}^{w^{*}}
		\text{ and }
		\varphi \circledast B_{\A^{**}} = \overline{ \varphi \circledast B_{\A}}^{w^{*}}.
		$$
	\end{cor}
	\begin{proof}
		We only consider the case of the left action, the other case is similar.
		Recall that $B_{\A^{**}}$ is weak-star compact in virtue of the Banach-Alaoglu theorem.
		Since the action is (weak-star, weak-star) continuous, $B_{\A^{**}} \circledast \varphi$ is also weak-star compact.
		In particular, it is weak-star closed, so we have:
		$$
		B_{\A^{**}} \circledast \varphi = \overline{B_{\A^{**}} \circledast \varphi}^{w^{*}} \supseteq \overline{B_{\A} \circledast \varphi}^{w^{*}}.
		$$
		Conversely, in virtue of Goldstine's theorem (Fact \ref{fact:goldstines}), $\overline{B_{\A}}^{w^{*}} = B_{\A^{**}}$.
		Again, since the action is (weak-star, weak-star) continuous we have:
		$$
		B_{\A^{**}} \circledast \varphi = \Or_{\varphi}^{(l)} \left( \overline{B_{\A}}^{w^{*}}\right) \subseteq
		\overline{\Or_{\varphi}^{(l)}(B_{\A})}^{w^{*}} =
		\overline{B_{\A} \circledast \varphi}^{w^{*}}.
		$$
	\end{proof}
	
	\begin{lemma}
		The restricted orbit maps $\Qr_{\varphi}^{(l)}$ and $\Qr_{\varphi}^{(r)}$ are (weak, weak) continuous.
		Specifically, if $\mu \in \A^{**}, \varphi \in \A^{*}$ and $\{a_{\lambda}\}_{\lambda \in \Lambda} \subseteq \A$ converges weakly to $a \in \A$ then:
		\begin{align*}
			\lim\limits_{\lambda \in \Lambda} \langle \mu, a_{\lambda} \circledast \varphi\rangle & = 
			\langle \mu, a \circledast \varphi\rangle\\
			\lim\limits_{\lambda \in \Lambda} \langle \mu, \varphi \circledast a_{\lambda}\rangle & = 
			\langle \mu, \varphi \circledast a\rangle.
		\end{align*}
	\end{lemma}
	\begin{proof}
		A consequence of Fact \ref{fact:bidual_action_properties} and Fact \ref{fact:adjoint_is_weak_and_weak_star_cont}.
	\end{proof}

	\begin{defin}
		For every $\mu \in \A^{**}$ we define $\arI_{\mu}^{(l)}, \arI_{\mu}^{(r)}, \arII_{\mu}^{(l)}, \arII_{\mu}^{(r)}\colon \A^{**} \to \A^{**}$ via:
		\begin{align*}
			\arI_{\mu}^{(l)}(\nu) := \nu \arI \mu,&\ 
			\arI_{\mu}^{(r)}(\nu) := \mu \arI \nu,\\ 
			\arII_{\mu}^{(l)}(\nu) := \nu \arII \mu,&\ 
			\arII_{\mu}^{(r)}(\nu) := \mu \arII \nu. 
		\end{align*}
	\end{defin}
	
	\begin{f}\cite[Thm.~1.4.2]{palmer}\label{lemma:bidual_ball_is_left_right_topological_semigroup}~
		
		\ben
		\item $\arI$ turns $(\A^{**}, w^{*})$ into a left topological semigroup. 
		Namely, the map $\arI_{\mu}^{(l)}$ sending $\nu$ to $\nu \arI \mu$ is weak-star continuous for every $\mu \in \A^{**}$.
		\item $\arII$ turns $(\A^{**}, w^{*})$ into a right topological semigroup. 
		Namely, the map $\arII_{\mu}^{(r)}$ sending $\nu$ to $\mu \arII \nu $ is weak-star continuous for every $\mu \in \A^{**}$.
		\een
	\end{f}
	
	In virtue of Fact \ref{lemma:arens_properties}, $B_{\A^{**}}$ is closed under both Arens products.
	Using the Banach–Alaoglu theorem, we get the following corollary.
	\begin{cor}~
		
		\ben
		\item $(B_{\A^{**}}, \arI, w^{*})$ is a compact left topological semigroup.
		\item $(B_{\A^{**}}, \arII, w^{*})$ is a compact right topological semigroup.
		\een
	\end{cor}

	\subsection{Classes of Functionals on Banach Algebras}\label{subsection:classes_of_functionals}
	Fragmentability, and fragmented functions had been studied in a variety of settings (multi-valued functions, uniform spaces, etc.).
	Here we show a specific (but equivalent) version which is most relevant.
	Other approaches can be found in \cite{JR,JOPV,Me-fr98}.
	\begin{defin} \label{def:fr} \cite{JOPV,Me-fr98} 
		Let $(X,\tau)$ be a topological space.
		A (not necessarily continuous) function $f\colon X \to M$ to a metric space $(M, d)$ is said to be \emph{fragmented} if for every nonempty subset $A$ of $X$ and every $\eps > 0$ there exists an open subset $O$ of $X$ such that $O \cap
		A$ is nonempty and $\diam_{d} \left(f(O \cap A)\right) <\eps$.
	\end{defin}
	\begin{defin} \label{d:fr-family} \
		Let $(X, \tau)$ be a topological space and let $\mathcal{F} \subseteq \R^{X}$ be a family of functions.
		It is said the $\mathcal{F}$ is:
		\ben
		\item\cite{GM1} A \emph{fragmented family} if the condition of Definition
		\ref{def:fr} holds simultaneously for all $f \in \mathcal F$.
		That is, $\diam f(O \cap A) < \eps$  for every $f \in \mathcal{
			F}$.
		\item \cite{GM-rose} An \emph{eventually fragmented family} if every 
		sequence in $F$ has a subsequence which is a fragmented family on $X$.
		\een
	\end{defin} 
	\begin{defin}
		Let $X$ be a set and $\{f_{n}\}_{n \in \N} \subseteq \R^{X}$ be a bounded family of function.
		It is said that $\{f_{n}\}_{n \in \N}$ is \emph{independent} if there are $a < b \in \R$ such that for every finite, disjoint $N, M \subseteq \N$ we can find $x \in X$ satisfying:
		$$
		  \forall n \in N: f_{n}(x) \leq a \text{ and } \forall m \in M: f_{m}(x) \geq b.
		$$
	\end{defin}
	We will now give a few equivalent definitions for Asplund and tame bounded subsets.
	Further details could be found in \cite[p.~22]{Fabian1997} and \cite[Thm.~2.6]{GM-MTame} in the Asplund case.
	For the tame case, the reader is referred to the summary in \cite[Thm.~2.4]{GM-MTame}. 
	
	\begin{defin}\label{defin:asplund_tame_subsets}
		Let $V$ be a Banach space and $C \subseteq V$ be a bounded subset.
		$C$ is said to be:
		\ben
		\item \emph{Asplund} if one of the following equivalent conditions holds:
		\ben
		\item $C$ is fragmented as a family of functions over $(B_{V^{*}}, w^{*})$.
		\item For every countable subset $D \subseteq C$, the pseudo-metric space $(B_{V^{*}}, \rho_{D})$ is separable, where:
		$
		\rho_{D}(\varphi, \psi) :=
		\sup\limits_{x \in D} \lvert \langle \varphi - \psi, x\rangle \rvert.
		$
		\een
		\item \emph{Tame} if one of the following equivalent conditions holds:
		\ben
		\item $C$ contains no $\ell^{1}$-sequence.
		\item Every sequence in $C$ contains a weak-Cauchy subsequence.
		\item $C$ is an eventually fragmented family over $(B_{V^{*}}, w^{*})$.
		\item There are no independent sequences $\{x_{n}\}_{n \in \N} \subseteq C$ over $B_{V^{*}}$.
		\een
		\een
	\end{defin}

	\begin{f} \label{fact:image_of_small_subset}
		Let $T\colon V \to W$ be a bounded operator between Banach spaces.
		If $A \subseteq V$ is some bounded Asplund (resp. tame) subset, then so is $T(A)$.
	\end{f}
	\begin{proof}
		In \cite[Lemma~3.7]{TameLCS} they prove this result for the more general notion of bornological classes, which includes Asplund and tame subsets \cite[Prop.~3.2 and Remark~4.14]{TameLCS}.
	\end{proof}
	\begin{defin}
		A Banach space $V$ is an \emph{Asplund} space if the dual of every separable subset is separable.
		A Banach space $V$ is said to be \emph{Rosenthal} if there are no $\ell^{1}$ sequences in $B_{V}$.
	\end{defin}
	\begin{f}~
		\ben
			\item \cite{NP} A Banach space $V$ is Asplund if and only if $B_{V}$ is an Asplund subset.
			\item \cite{GM-MTame} A Banach space $V$ is Rosenthal if and only if $B_{V}$ is a tame subset. 
		\een
	\end{f}
	Another important property is that reflexive spaces are necessarily Asplund, which are necessarily tame.
	In other words:
	$$
	\text{Reflexive} \subseteq \text{Asplund} \subseteq \text{Tame}.
	$$
	
	\begin{defin}\cite[Def.~7.10]{TameLCS}
		Let $V$ be a Banach space.
		A weak-star compact subset $M \subseteq V$ is said to be \emph{co-tame} if $B_{V}$ is tame as a family of functions over $M$.
	\end{defin}
	We will use the following localization of Haydon's theorem.
	\begin{f}\cite[Thm.~10.12]{TameLCS} \label{fact:co_tame_milman}
		Let $V$ be a Banach space and $M \subseteq V^{*}$ be a convex weak-star compact subset.
		Then $M$ is co-tame if and only if
		$$
		\overline{\co}^{w^{*}}(N) = \overline{\co}(N),
		$$
		for every weak-star closed subset $N \subseteq M$.
	\end{f}
	We will use this fact in conjunction of the following property
	\begin{lemma} \label{lemma:separable_implies_cotame}
		Let $V$ be a Banach space and let $M \subseteq V^{*}$ be a weak-star compact subset of the dual.
		If $M$ is norm separable, then it is co-tame.
	\end{lemma}
	\begin{proof}
		By contradiction, assume that $M$ is not co-tame.
		Therefore, we can find a sequence $\{x_{n}\}_{n \in \N} \subseteq B_{V}$ and $a < b \in \R$ such that:
		$$
		\forall \alpha \in \{0, 1\}^{\N} \, \exists \varphi_{\alpha} \in M\, \forall n \in \N: \langle \varphi_{\alpha}, x_{n}\rangle \in \begin{cases}
			[b, \infty) & \alpha_{n} = 1\\
			(-\infty, a] & \alpha_{n} = 0
		\end{cases}.
		$$
		Write $0 < \eps := b - a$.
		It is easy to see that for every $\alpha \neq \beta \in \{0, 1\}^{\N}$ we have:
		\begin{align*}
			\lVert \varphi_{\alpha} - \varphi_{\beta} \rVert & = 
			\sup\limits_{x \in B_{V}} \lvert \langle \varphi_{\alpha}, x\rangle - \langle \varphi_{\beta}, x\rangle\rvert\\
			& \geq 
			\sup\limits_{n \in \N} \lvert \langle \varphi_{\alpha}, x_{n}\rangle - \langle \varphi_{\beta}, x_{n}\rangle\rvert\\
			& \geq b - a = \eps > 0.
		\end{align*}
		Therefore, $\{\varphi_{\alpha}\}_{\alpha \in \{0, 1\}^{\N}}$ is a discrete subset of $M$ whose cardinality is $2^{\aleph_{0}}$, in contradiction to $M$ being separable.
	\end{proof}
	Weakly almost periodic functionals over Banach algebras have many interesting properties (see for example \cite{burckel} and \cite{Ulger}).
	In \cite{TameFunc}, some generalizations of this concept were proposed, namely Asplund and tame functionals.
	Equivalent definitions, and specific properties in the case of the group algebra were further studied in \cite{meTameFunc}.
	The following is one set of such equivalent definitions suitable for our needs.
	For a more detailed exposition the reader is referred to \cite{meTameFunc}.
	\begin{defin}
		A functional $\varphi \in \A^{*}$ is said to be:
		\ben
		\item \emph{weakly almost periodic} if $B_{\A} \circledast \varphi$ is weakly relatively compact in $\A^{*}$.
		\item \emph{left Asplund} if $B_{\A} \circledast \varphi$ is an Asplund subset of $\A^{*}$. 
		\item \emph{right Asplund} if $\varphi \circledast B_{\A}$ is an Asplund subset of $\A^{*}$.
		\item \emph{left tame} if $B_{\A} \circledast \varphi$ is a tame subset of $\A^{*}$.
		\item \emph{right tame} if $\varphi \circledast B_{\A}$ is a tame subset of $\A^{*}$.
		\een
		We will write $\wap(\A), \asp^{(l)}(\A), \asp^{(r)}(\A), \tame^{(l)}(\A), \tame^{(r)}(\A)$ for the appropriate subsets of $\A^{*}$.
	\end{defin}
	\begin{remark}
		Note that we do not distinguish between ``left" and ``right" weakly almost periodic functionals since those definitions are equivalent. See \cite[Cor.~1.12]{burckel} for the case of semigroup algebras or Proposition \ref{prop:wap_equivalence} for a detailed explanation.
		As far as we are aware, the relationship between left and right Asplund or tame functionals is still an open question \cite[Quest.~7.2]{meTameFunc}.
	\end{remark}
	\begin{remark}
		Unlike \cite{meTameFunc}, we include in these definitions all of the functionals in $\A^{*}$ rather than the left or right uniformly continuous ones.
		We do not use any property that rely on this restriction.
		This discussion is of course redundant in the case of weakly almost periodic functionals, since every such functional is uniformly continuous \cite[Thm.~2]{Ulger}.
	\end{remark}

	Recall that an operator $T\colon V\to W$ between two Banach spaces is weakly relatively compact/Asplund/tame if $T(B_{V})$ is weakly relatively compact/Asplund/tame in $W$, respectively.
	The following fact is a consequence of the Davis-Figiel-Johnson-Pe\l{}czy\'{n}ski factorization which can be found in \cite{DFJP}.
	Its application to reflexive, Asplund and Rosenthal spaces can be found in \cite[Cor.~3]{DFJP}, \cite[Sec.~1.3]{Fabian1997}, \cite[Thm.~6.3]{GM-rose}, respectively.
	\begin{f}
		Let $T\colon V \to W$ be an operator between Banach spaces.
		Then $T(B_{V})$ is relatively weakly compact/Asplund/tame if and only if $T$ factors through a reflexive/Asplund/Rosenthal Banach space.
	\end{f}
	As with many other important applications, this factorization turns out to be crucial for the results of this paper.
	Specifically, together with Fact \ref{lemma:adjoint_representation} we get the following important corollary.
	\begin{cor}\label{cor:reflexive_asplund_rosenthal_factorization}
		Let $\varphi \in \A^{*}$ be a right weakly almost periodic/Asplund/tame functional.
		Then $\Qr_{\varphi}^{(r)}$ factors through a reflexive/Asplund/Rosenthal Banach space $V$ respectively.
		Therefore, there exists bounded maps $T\colon \A\to V$ and $S\colon V \to \A^{*}$ such that $\Qr_{\varphi}^{(r)} = S\circ T$. 
		As a consequence
		$\Or_{\varphi}^{(l)} = T^{*}\circ S^{*}$ and the following diagram commutes.
		\[
		\begin{array}{cc}
			\begin{tikzcd}
				\A && {\A^{*}} \\
				& V
				\arrow["{\Qr_{\varphi}^{(r)}}", from=1-1, to=1-3]
				\arrow["T"', from=1-1, to=2-2]
				\arrow["S"', from=2-2, to=1-3]
			\end{tikzcd} &
			\begin{tikzcd}
				{\A^{*}} && {\A^{**}} \\
				& {V^{*}}
				\arrow["{\Or_{\varphi}^{(l)}}"', from=1-3, to=1-1]
				\arrow["{S^{*}}", from=1-3, to=2-2]
				\arrow["{T^{*}}", from=2-2, to=1-1]
			\end{tikzcd}
		\end{array}
		\]
		An analogous statement could be made about left weakly almost periodic/Asplund/tame functionals.
	\end{cor}

	\section{Weakly Almost Periodic Functionals}\label{section:wap_functionals}
	This section is a summary of many known results about weakly almost periodic functionals.
	In the next section we will use a similar analysis to give a detailed exploration of the properties of Asplund and tame functionals.
	\begin{f}[Grothendieck's criterion] \cite[Thm.~17.12]{KellyNamioka}\label{fact:grothendiecks_criterion}
		Let $V$ be a Banach space and $C \subseteq V$ be a bounded subset.
		Then $C$ is relatively weakly compact if and only if 
		$$
		\lim\limits_{n \in \N} \lim\limits_{m \in \N} \langle \varphi_{n}, x_{m}\rangle = \lim\limits_{m \in \N} \lim\limits_{n \in \N} \langle \varphi_{n}, x_{m}\rangle
		$$
		whenever $\{x_{m}\}_{m \in \N} \subseteq C$ and $\{\varphi_{n}\}_{n \in \N} \subseteq B_{V^{*}}$ such that both limits exist.
	\end{f}
	The following theorem is an amalgamation of many known results.
	Most of them could be reconstructed from a global version (treating Arens regularity rather than individual weakly almost periodic functionals) found in \cite[Thm.~1.4.11]{palmer}.
	A proof for the equivalent of \ref{prop:wap_equivalence:commutator} and \ref{prop:wap_equivalence:left_weakly_compact} can be found in \cite[Thm.~4.2]{pymConvFunc}.
	We include the proofs for the convenience of the reader.
	\begin{prop}\label{prop:wap_equivalence}
		Let $\varphi \in \A^{*}$ be some functional, then the following are equivalent:
		\ben
		\item $\varphi$ is weakly almost periodic (meaning that $B_{\A} \circledast \varphi$ is relatively weakly compact).
		\label{prop:wap_equivalence:left_weakly_compact}
		\item $\varphi \circledast B_{\A}$ is relatively weakly compact.
		\label{prop:wap_equivalence:right_weakly_compact}
		\item $B_{\A^{**}} \circledast \varphi$ is weakly compact.
		\label{prop:wap_equivalence:left_arens_orbit_weakly_compact}
		\item $\varphi \circledast B_{\A^{**}}$ is weakly compact.
		\label{prop:wap_equivalence:right_arens_orbit_weakly_compact}
		\item The orbit map $\Or_{\varphi}^{(l)}$ is (weak-star, weak) continuous.
		\label{prop:wap_equivalence:left_orbit_weak_continuous}
		\item The orbit map $\Or_{\varphi}^{(r)}$ is (weak-star, weak) continuous.
		\label{prop:wap_equivalence:right_orbit_weak_continuous}
		\item If $\{\mu_{\lambda} \}_{\lambda \in \Lambda} \subseteq \A^{**}$ converges to $\mu \in \A^{**}$ in the weak-star topology, then for every $\nu \in \A^{**}$:
		$$
		\lim\limits_{\lambda \in \Lambda} \langle \nu \arI \mu_{\lambda}, \varphi\rangle = \langle \nu \arI \mu, \varphi\rangle.
		$$
		In other words, $\arI_{\nu}^{(r)}$ is (weak-star, $\rho_{\varphi}$) continuous.
		\label{prop:wap_equivalence:first_arens_continuity}
		\item If $\{\mu_{\lambda} \}_{\lambda \in \Lambda} \subseteq \A^{**}$ converges to $\mu \in \A^{**}$ in the weak-star topology, then for every $\nu \in \A^{**}$:
		$$
		\lim\limits_{\lambda \in \Lambda} \langle \mu_{\lambda} \arII \nu, \varphi\rangle = \langle \mu \arII \nu, \varphi\rangle.
		$$
		In other words, $\arII_{\nu}^{(l)}$ is (weak-star, $\rho_{\varphi}$) continuous.
		\label{prop:wap_equivalence:second_arens_continuity}
		\item For every $\mu, \nu \in \A^{**}$:
		$$
		\langle \mu \arI \nu - \mu \arII \nu, \varphi\rangle = 0.
		$$
		\label{prop:wap_equivalence:commutator}
		\item For every $\{a_{n}\}_{n \in \N} \subseteq B_{\A}$ and $\{\mu_{m}\}_{m \in\N} \subseteq B_{\A^{**}}$ we have:
		$$
		\lim\limits_{n \in \N}\lim\limits_{m \in \N} \langle \mu_{m}, a_{n} \circledast \varphi\rangle = 
		\lim\limits_{m \in \N}\lim\limits_{n \in \N} \langle \mu_{m}, a_{n} \circledast \varphi\rangle
		$$
		whenever both limits exist.
		\label{prop:wap_equivalence:grothendiecks}
		\een
	\end{prop}
	\begin{proof}
		The proof will follow the following scheme:
		$$
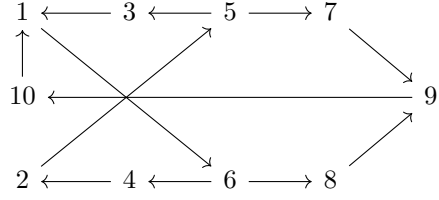
\begin{tikzcd}
			\ref{prop:wap_equivalence:left_weakly_compact}& \ref{prop:wap_equivalence:left_arens_orbit_weakly_compact} & \ref{prop:wap_equivalence:left_orbit_weak_continuous} & \ref{prop:wap_equivalence:first_arens_continuity} \\
			\ref{prop:wap_equivalence:grothendiecks} &&&& \ref{prop:wap_equivalence:commutator} \\
			\ref{prop:wap_equivalence:right_weakly_compact} & \ref{prop:wap_equivalence:right_arens_orbit_weakly_compact} & \ref{prop:wap_equivalence:right_orbit_weak_continuous} & \ref{prop:wap_equivalence:second_arens_continuity} \\
			\arrow[from=1-1, to=3-3]
			\arrow[from=1-2, to=1-1]
			\arrow[from=1-3, to=1-2]
			\arrow[from=1-3, to=1-4]
			\arrow[from=1-4, to=2-5]
			\arrow[from=2-1, to=1-1]
			\arrow[from=2-5, to=2-1]
			\arrow[from=3-1, to=1-3]
			\arrow[from=3-2, to=3-1]
			\arrow[from=3-3, to=3-2]
			\arrow[from=3-3, to=3-4]
			\arrow[from=3-4, to=2-5]
		\end{tikzcd}
		$$
		\ben
		\item[\ref{prop:wap_equivalence:left_weakly_compact} $\Rightarrow$ \ref{prop:wap_equivalence:right_orbit_weak_continuous}] 
		By definition, $\Qr_{\varphi}^{(l)}(B_{\A})$ is relatively weakly compact.
		Applying Corollary \ref{cor:reflexive_asplund_rosenthal_factorization}, we can find a reflexive space $V$ and bounded maps $T\colon \A \to V$ and $S \colon V\to \A^{*}$ such that $\Qr_{\varphi}^{(l)} = S \circ T$ and $\Or_{\varphi}^{(r)} = T^{*} \circ S^{*}$.
		By Fact \ref{fact:adjoint_is_weak_and_weak_star_cont}, $S^{*}$ is (weak-star, weak-star) continuous, and $T^{*}$ is (weak, weak) continuous.
		Moreover, the weak and weak-star topologies on $V^{*}$ are the same because $V$ is reflexive.
		Therefore, $T^{*}$ is (weak-star, weak) continuous.
		Together we conclude that $\Or_{\varphi}^{(r)}$ is (weak-star, weak) continuous.
		\item[\ref{prop:wap_equivalence:right_weakly_compact} $\Rightarrow$ \ref{prop:wap_equivalence:left_orbit_weak_continuous}] Proven analogously.
		\item[\ref{prop:wap_equivalence:left_arens_orbit_weakly_compact} $\Rightarrow$ \ref{prop:wap_equivalence:left_weakly_compact}] Trivial.
		\item[\ref{prop:wap_equivalence:right_arens_orbit_weakly_compact} $\Rightarrow$ \ref{prop:wap_equivalence:right_weakly_compact}] Trivial.
		\item[\ref{prop:wap_equivalence:left_orbit_weak_continuous} $\Rightarrow$ \ref{prop:wap_equivalence:left_arens_orbit_weakly_compact}] If $\Or_{\varphi}^{(l)}$ is (weak-star, weak) continuous then $B_{\A^{**}}$ being weak-star compact (Banach–Alaoglu theorem) implies that $\Or_{\varphi}^{(l)}(B_{\A^{**}}) \subseteq \A^{*}$ is weakly compact.
		However, note that:
		$$
		B_{\A^{**}} \circledast \varphi = \Or_{\varphi}^{(l)}(B_{\A^{**}}).
		$$
		\item[\ref{prop:wap_equivalence:right_orbit_weak_continuous} $\Rightarrow$ \ref{prop:wap_equivalence:right_arens_orbit_weakly_compact}] Proven analogously.
		\item[\ref{prop:wap_equivalence:left_orbit_weak_continuous} $\Rightarrow$ \ref{prop:wap_equivalence:first_arens_continuity}] Suppose that $\Or_{\varphi}^{(l)}$ is (weak-star, weak) continuous and $\{\mu_{\lambda}\}_{\lambda \in \Lambda} \subseteq \A^{**}$ converges in the weak-star topology to $\mu \in \A^{**}$.
		Then for every $\nu \in \A^{**}$:
		$$
		\lim\limits_{\lambda \in \Lambda} \langle \nu \arI \mu_{\lambda}, \varphi\rangle =
		\lim\limits_{\lambda \in \Lambda} \langle \nu, \mu_{\lambda} \circledast \varphi\rangle =
		\lim\limits_{\lambda \in \Lambda} \langle \nu, \Or_{\varphi}^{(l)}(\mu_{\lambda})\rangle =
		\langle \nu, \Or_{\varphi}^{(l)}(\mu)\rangle = 
		\langle \nu \arI \mu, \varphi\rangle.
		$$
		\item[\ref{prop:wap_equivalence:right_orbit_weak_continuous} $\Rightarrow$ \ref{prop:wap_equivalence:second_arens_continuity}] Proven analogously.
		\item[\ref{prop:wap_equivalence:first_arens_continuity} $\Rightarrow$ \ref{prop:wap_equivalence:commutator}] Suppose that $\nu, \mu \in \A^{**}$.
		Applying Goldstine's theorem (Fact \ref{fact:goldstines}), we can find a net $\{a_{\lambda} \}_{\lambda \in \Lambda} \subseteq \A$ which converges to $\nu$ in the weak-star topology.
		Now we have:
		\begin{align*}
			\langle \mu \arI \nu - \mu \arII \nu, \varphi\rangle & \underset{\ref{lemma:bidual_ball_is_left_right_topological_semigroup}}{=}
			\lim\limits_{\lambda \in \Lambda} \langle \mu \arI \nu - \mu \arII a_{\lambda}, \varphi\rangle\\
			& \underset{\ref{prop:wap_equivalence:first_arens_continuity}}{=}
			\lim\limits_{\lambda \in \Lambda} \langle \mu \arI a_{\lambda} - \mu \arII a_{\lambda}, \varphi\rangle\\
			& \underset{\ref{lemma:arens_properties}}{=} \lim\limits_{\lambda \in \Lambda} \langle \mu a_{\lambda} - \mu a_{\lambda}, \varphi\rangle = 0.
		\end{align*}
		\item[\ref{prop:wap_equivalence:second_arens_continuity} $\Rightarrow$ \ref{prop:wap_equivalence:commutator}] Proven analogously.
		\item[\ref{prop:wap_equivalence:commutator} $\Rightarrow$ \ref{prop:wap_equivalence:grothendiecks}] Let $\{a_{n}\}_{n \in \N} \subseteq B_{\A}$ and $\{ \mu_{m}\}_{m \in \N} \subseteq B_{\A^{**}}$ such that the limits
		$$
		\lim\limits_{n \in \N}\lim\limits_{m \in \N} \langle \mu_{m}, a_{n} \circledast \varphi\rangle \text{ and }
		\lim\limits_{m \in \N}\lim\limits_{n \in \N} \langle \mu_{m}, a_{n} \circledast \varphi\rangle
		$$
		exist.
		Applying the Banach–Alaoglu theorem, we can find a subnet $\{n(\lambda)\}_{\lambda \in \Lambda} \subseteq \N$ such that $\{\mu_{n(\lambda)}\}_{\lambda \in \Lambda}$ and $\{a_{n(\lambda)}\}_{\lambda \in \Lambda}$ converges in the weak-star topology to $\mu, \nu \in B_{\A^{**}}$ respectively.
		We now have:
		\begin{align*}
			\lim\limits_{n \in \N}\lim\limits_{m \in \N} \langle \mu_{m}, a_{n} \circledast \varphi\rangle & = 
			\lim\limits_{\lambda \in \Lambda}\lim\limits_{\kappa \in \Lambda} \langle \mu_{n(\kappa)}, a_{n(\lambda)} \circledast \varphi\rangle \\
			& = \lim\limits_{\lambda \in \Lambda}\lim\limits_{\kappa \in \Lambda} \langle \mu_{n(\kappa)} \arI a_{n(\lambda)}, \varphi\rangle \\
			& \underset{\ref{lemma:bidual_ball_is_left_right_topological_semigroup}}{=}
			\lim\limits_{\lambda \in \Lambda} \langle \mu \arI a_{n(\lambda)}, \varphi\rangle \\
			& \underset{\ref{lemma:arens_properties}}{=}
			\lim\limits_{\lambda \in \Lambda} \langle \mu \arII a_{n(\lambda)}, \varphi\rangle \\
			& \underset{\ref{lemma:bidual_ball_is_left_right_topological_semigroup}}{=} 
			\langle \mu \arII \nu, \varphi\rangle,
		\end{align*}
		\begin{align*}
			\lim\limits_{m \in \N}\lim\limits_{n \in \N} \langle \mu_{m}, a_{n} \circledast \varphi\rangle & =
			\lim\limits_{\kappa \in \Lambda}\lim\limits_{\lambda \in \Lambda} \langle \mu_{n(\kappa)}, a_{n(\lambda)} \circledast \varphi\rangle \\
			& = \lim\limits_{\kappa \in \Lambda}\lim\limits_{\lambda \in \Lambda} \langle \mu_{n(\kappa)} \arI a_{n(\lambda)}, \varphi\rangle \\
			& \underset{\ref{lemma:arens_properties}}{=} \lim\limits_{\kappa \in \Lambda}\lim\limits_{\lambda \in \Lambda} \langle \mu_{n(\kappa)} \arII a_{n(\lambda)}, \varphi\rangle \\
			& \underset{\ref{lemma:bidual_ball_is_left_right_topological_semigroup}}{=}
			\lim\limits_{\kappa \in \Lambda}\langle \mu_{n(\kappa)} \arII \nu, \varphi\rangle \\
			& \underset{\ref{prop:wap_equivalence:commutator}}{=}
			\lim\limits_{\kappa \in \Lambda}\langle \mu_{n(\kappa)} \arI \nu, \varphi\rangle \\
			& \underset{\ref{lemma:bidual_ball_is_left_right_topological_semigroup}}{=}
			\langle \mu \arI \nu, \varphi\rangle\\
			& \underset{\ref{prop:wap_equivalence:commutator}}{=} 
			\langle \mu \arII \nu, \varphi\rangle \\
			& = \lim\limits_{n \in \N}\lim\limits_{m \in \N} \langle \mu_{m}, a_{n} \circledast \varphi\rangle.
		\end{align*}
		\item[\ref{prop:wap_equivalence:grothendiecks} $\Rightarrow$ \ref{prop:wap_equivalence:left_weakly_compact}] Follows from Grothendieck's criterion for relative weak compactness (Fact \ref{fact:grothendiecks_criterion}).
		\een
	\end{proof}
	
	\begin{cor} \label{cor:weak_star_orbit_is_the_same_for_wap}
		If $\varphi \in \A^{*}$ is weakly almost periodic, then:
		$$
		B_{\A^{**}} \circledast \varphi = \overline{B_{\A} \circledast \varphi} \text{ and }
		\varphi \circledast B_{\A^{**}} = \overline{\varphi \circledast B_{\A}}.
		$$
	\end{cor}
	\begin{proof}
		Using the previous theorem, we conclude that $\overline{B_{\A} \circledast \varphi}^{w}$ is weakly compact, hence weak-star compact and thus weak-star closed.
		In other words, we have:
		$$
		\overline{B_{\A} \circledast \varphi}^{w} = \overline{B_{\A} \circledast \varphi}^{w^{*}} 
		\underset{\ref{cor:dual_ball_action_is_weak_star_closure}}{=} B_{\A^{**}} \circledast \varphi.
		$$ 
		Also, $B_{\A} \circledast \varphi$ is convex, so Mazur's theorem (Fact \ref{fact:mazur}) implies that $\overline{B_{\A} \circledast \varphi} = \overline{B_{\A} \circledast \varphi}^{w}$.
		Thus:
		$$
		\overline{B_{\A} \circledast \varphi} = \overline{B_{\A} \circledast \varphi}^{w} = B_{\A^{**}} \circledast \varphi.
		$$
		The proof of the analogous right version is similar.
	\end{proof}
	\section{Asplund and Tame Functionals} \label{section:asplund_and_tame_functionals}

	\begin{defin}
		Let $C \subseteq \A^{*}$ be some bounded subset.
		We define the $C$-seminorm $\lVert \cdot \rVert_{C}$ on $\A^{**}$ via:
		$$
		\forall \mu \in \A^{**}: \lVert \mu \rVert_{C} := \sup\limits_{\varphi \in C} \lvert \langle \mu, \varphi\rangle\rvert.
		$$
		Also, if $\varphi \in \A^{*}$, we define the induced left and right seminorms via:
		$$
		\lVert \cdot \rVert_{\varphi}^{(l)} := \lVert \cdot \rVert_{B_{\A}\circledast \varphi},\ 
		\lVert \cdot \rVert_{\varphi}^{(r)} := \lVert \cdot \rVert_{\varphi \circledast B_{\A}}. 
		$$
		Note that $B_{\A} \circledast \varphi$ and $\varphi \circledast B_{\A}$ are indeed bounded in virtue of Fact \ref{fact:bidual_action_properties}.
	\end{defin}
	
	\begin{lemma} \label{lemma:orbit_maps_are_equivalent_to_uniform_convergence}
		The left orbit map $\Or_{\varphi}^{(l)}\colon (\A^{**}, \lVert \cdot \rVert_{\varphi}^{(r)}) \to (\A^{*}, \lVert \cdot \rVert)$ is an isometry of seminorms.
		The right orbit map $\Or_{\varphi}^{(r)}\colon (\A^{**}, \lVert \cdot \rVert_{\varphi}^{(l)}) \to (\A^{*}, \lVert \cdot \rVert)$ is also an isometry of seminorms.
	\end{lemma}
	\begin{proof}
		For every $\mu \in \A^{**}$ we have:
		\begin{align*}
			\lVert \Or_{\varphi}^{(l)}(\mu) \rVert & =
			\lVert \mu \circledast \varphi \rVert\\
			& = \sup\limits_{b \in B_{\A}} \lvert \langle \mu\circledast \varphi, b\rangle \rvert\\
			& = \sup\limits_{b \in B_{\A}} \lvert \langle \mu, \varphi \circledast b\rangle\rvert\\
			& = \sup\limits_{\psi \in \varphi \circledast B_{\A}} \lvert \langle \mu, \psi\rangle\rvert = \lVert \mu \rVert_{\varphi}^{(r)}.
		\end{align*}
		\begin{align*}
			\lVert \Or_{\varphi}^{(r)}(\mu) \rVert & =
			\lVert \varphi \circledast \mu \rVert\\
			& = \sup\limits_{b \in B_{\A}}\lvert \langle \varphi \circledast \mu, b\rangle\rvert\\
			& = \sup\limits_{b \in B_{\A}} \lvert \langle \mu, b \circledast \varphi\rangle\rvert\\
			& = \sup\limits_{\psi \in B_{\A} \circledast \varphi} \lvert \langle \mu, \psi\rangle\rvert = \lVert \mu \rVert_{\varphi}^{(l)}.
		\end{align*}
	\end{proof}
	\begin{thm} \label{thm:asplund_functionals}
		Let $\varphi \in \A^{*}$, then the following are equivalent:
		\ben
		\item $\varphi$ is right Asplund, meaning that $\varphi \circledast B_{\A}$ is fragmented over $B_{\A^{**}}$.
		\label{thm:asplund_functionals:asplund_functional}
		\item The orbit map $\Or_{\varphi}^{(l)}$ is (weak-star, norm) fragmented over $B_{\A^{**}}$.
		\label{thm:asplund_functionals:fragmented_orbit_map}
		\item The family of functions $\{\arI_{\nu}^{(r)}\}_{\nu \in B_{\A^{**}}}$ is (weak-star, $\rho_{\varphi}$) fragmented over $B_{\A^{**}}$.
		\label{thm:asplund_functionals:fragmented_first_arens}
		\een
		Analogous statements hold for left Asplund functionals. 
	\end{thm}
	\begin{proof}
		First, note that by applying Goldstine's theorem (Fact \ref{fact:goldstines}) and left weak-star continuity of the Arens products (Fact \ref{lemma:bidual_ball_is_left_right_topological_semigroup}), we get:
		\begin{equation*} 
			\sup\limits_{\nu \in B_{\A^{**}}} \langle \nu \arI \mu, \varphi\rangle = 
			\sup\limits_{a \in B_{\A}} \langle a \arI \mu, \varphi\rangle.
		\end{equation*}
		Now, by definition:
		\ben
		\item $\varphi \circledast B_{\A}$ is fragmented over $B_{\A^{**}}$ if and only if for every $\eps > 0$ and non-empty $C \subseteq B_{\A^{**}}$, there exists weak-star open $O \subseteq \A^{**}$ such that $C \cap O \neq \emptyset$ and
		$$
		\sup\limits_{\psi \in \varphi \circledast B_{\A}} \diam_{\lvert \cdot \rvert}
		\left(\langle C \cap O, \psi \rangle\right) < \eps.
		$$
		\item The orbit map $\Or_{\varphi}^{(l)}$ is (weak-star, norm) fragmented over $B_{\A^{**}}$ if and only if for every $\eps > 0$ and non-empty $C \subseteq B_{\A^{**}}$ there exists weak-star open $O \subseteq \A^{**}$ such that $C \cap O \neq \emptyset$ and:
		$$
		\diam_{\lVert \cdot \rVert} \left(\Or_{\varphi}^{(l)} (C \cap O) \right) < \eps.
		$$
		\item the family $\left\{\arI_{\nu}^{(r)} \right\}_{\nu \in B_{\A^{**}}}$ is (weak-star, $\rho_{\varphi}$) fragmented over $B_{\A^{**}}$ if and only if for every $\eps > 0$ and non-empty $C \subseteq B_{\A^{**}}$, there exists weak-star open $O \subseteq \A^{**}$ such that $C \cap O \neq \emptyset$ and
		$$
		\sup\limits_{\nu \in B_{\A^{**}}} \diam_{\rho_{\varphi}} \left( \arI_{\nu}^{(r)} (C \cap O) \right) < \eps.
		$$
		\een
		Note that:
		\begin{align*}
			\diam_{\lVert \cdot \rVert} \left(\Or_{\varphi}^{(l)} (C \cap O) \right) & \underset{\ref{lemma:orbit_maps_are_equivalent_to_uniform_convergence}}{=} \diam_{\lVert \cdot \rVert_{\varphi}^{(r)}} \left( C \cap O \right) \\
			& = 
			\sup\limits_{\mu, \nu \in C\cap O} \lVert \mu - \nu \rVert_{\varphi}^{(r)}\\
			& = \sup\limits_{\mu, \nu \in C\cap O} \sup\limits_{\psi \in \varphi \circledast B_{\A}} \lvert \langle \mu - \nu, \psi \rangle\rvert\\
			& = \sup\limits_{\psi \in \varphi \circledast B_{\A}}\sup\limits_{\mu, \nu \in C\cap O} \lvert \langle \mu - \nu, \psi \rangle\rvert \\
			& = \sup\limits_{\psi \in \varphi \circledast B_{\A}}\diam_{\lvert \cdot \rvert} \left(\langle C\cap O, \psi \rangle\right).
		\end{align*}
		Therefore (\ref{thm:asplund_functionals:asplund_functional}) and (\ref{thm:asplund_functionals:fragmented_orbit_map}) are indeed equivalent.
		Moreover, 
		\begin{align*}
			\sup\limits_{\nu \in B_{\A^{**}}}\diam_{\rho_{\varphi}} \left( \arI_{\nu}^{(r)} (C \cap O) \right) & = 
			\sup\limits_{\nu \in B_{\A^{**}}}\sup\limits_{\mu, \kappa \in C \cap O} \left\lvert \left\langle \arI_{\nu}^{(r)} (\mu - \kappa), \varphi\right\rangle \right\rvert \\
			& = \sup\limits_{\nu \in B_{\A^{**}}}\sup\limits_{\mu, \kappa \in C \cap O} \lvert \langle \nu \arI (\mu - \kappa), \varphi\rangle\rvert\\
			& = \sup\limits_{\mu, \kappa \in C \cap O}\sup\limits_{\nu \in B_{\A^{**}}} \lvert \langle \nu \arI (\mu - \kappa), \varphi\rangle\rvert\\			
			& = \sup\limits_{\mu, \kappa \in C \cap O} \sup\limits_{a \in B_{\A}} \lvert \langle a \arI (\mu - \kappa), \varphi\rangle\rvert\\
			& \underset{\ref{lemma:arens_properties}}{=} \sup\limits_{\mu, \kappa \in C \cap O} \sup\limits_{a \in B_{\A}} \lvert \langle a \arII (\mu - \kappa), \varphi\rangle\rvert\\
			& = \sup\limits_{\mu, \kappa \in C \cap O} \sup\limits_{a \in B_{\A}} \lvert \langle \mu - \kappa, \varphi \circledast a\rangle\rvert\\
			& = \sup\limits_{a \in B_{\A}} \sup\limits_{\mu, \kappa \in C \cap O} \lvert \langle \mu - \kappa, \varphi \circledast a\rangle\rvert\\
			& = \sup\limits_{\psi \in \varphi \circledast B_{\A}} \diam_{\lvert \cdot \rvert} \left(\langle C \cap O, \psi\rangle\right).
		\end{align*}
		Thus, (\ref{thm:asplund_functionals:asplund_functional}) and (\ref{thm:asplund_functionals:fragmented_first_arens}) are also equivalent.
	\end{proof}

	The following lemma and then Theorem \ref{thm:separable_asplund_iff_separable_arens_orbit} are analogues of a key characteristic of Asplund Banach spaces: a Banach space $V$ is Asplund if and only if every separable subspace has a separable dual \cite[thm.~A]{stegal75}.
	\begin{lemma} \label{lemma:separable_implies_asplund}
		Let $\varphi \in \A^{*}$ be some functional.
		\ben
		\item If $B_{\A^{**}} \circledast \varphi$ is norm separable then $\varphi$ is right Asplund.
		\item If $\varphi \circledast B_{\A^{**}}$ is norm separable then $\varphi$ is left Asplund.
		\een
	\end{lemma}
	\begin{proof}
		We will only consider the right Asplund case.
		In virtue of Lemma \ref{lemma:orbit_maps_are_equivalent_to_uniform_convergence}, $B_{\A^{**}} \circledast \varphi$ is separable if and only if $(B_{\A^{**}}, \lVert \cdot \rVert_{\varphi}^{(r)})$ is separable.
		Also, for every countable $C \subseteq \varphi \circledast B_{\A}$ and $\mu \in \A^{**}$ we have:
		$$
		\rho_{C}(\mu) = \sup\limits_{\psi \in C} \lvert \langle \mu, \psi\rangle\rvert \leq \sup\limits_{\psi \in \varphi \circledast B_{\A}} \lvert \langle \mu, \psi\rangle\rvert = \lVert \mu \rVert_{\varphi}^{(r)}.
		$$
		So $(B_{\A^{**}}, \lVert \cdot \rVert_{\varphi}^{(r)})$ being separable implies that $(B_{\A^{**}}, \rho_{C})$ is separable for every countable $C \subseteq \varphi \circledast B_{\A}$.
		By definition \ref{defin:asplund_tame_subsets}, $\varphi \circledast B_{\A}$ is an Asplund subset, making $\varphi$ right Asplund, as required.
	\end{proof}
	
	\begin{thm} \label{thm:separable_asplund_iff_separable_arens_orbit}
		Let $\A$ be a separable Banach Algebra and let $\varphi \in \A^{*}$ be a functional.
		Then $B_{\A^{**}} \circledast \varphi$ is separable if and only if $\varphi$ is right Asplund.
	\end{thm}
	\begin{proof}
		First suppose that $\varphi$ is right Asplund.
		Let $D \subseteq B_{\A}$ be a countable dense subset.
		Write $\rho := \rho_{\varphi \circledast D}$.
		We claim that $\rho = \lVert \cdot \rVert_{\varphi}^{(r)}$.
		Indeed, clearly $\rho \leq \lVert \cdot \rVert_{\varphi}^{(r)}$.
		To see the converse, note that for every $\varphi \circledast a \in \varphi \circledast B_{\A}$ there exist a sequence $\{a_{n}\}_{n \in \N} \subseteq D$ such that:
		$$
		a = \lim\limits_{n \in \N} a_{n}.
		$$
		Therefore, for every $\mu \in \A^{**}$:
		$$
		\lvert \langle \mu, \varphi \circledast a\rangle \rvert = \lim\limits_{n \in \N} \lvert \langle \mu, \varphi \circledast a_{n}\rangle\rvert \leq 
		\sup\limits_{b \in D} \lvert \langle \mu, \varphi \circledast b\rangle\rvert = \rho(\mu). 
		$$
		This is true for every $\varphi \circledast a \in \varphi \circledast B_{\A}$ so:
		$$
		\lVert \mu \rVert_{\varphi}^{(r)} = 
		\sup\limits_{a \in B_{\A}} \lvert \langle\mu, \varphi \circledast a\rangle\rvert
		\leq \rho(\mu).
		$$
		Clearly, $\varphi \circledast D$ is also countable, so by definition (\ref{defin:asplund_tame_subsets}) of Asplund subsets, $(B_{\A^{**}}, \rho_{\varphi \circledast D})$ is separable.
		In other words, $(B_{\A^{**}}, \lVert \cdot \rVert_{\varphi}^{(r)})$ is separable.
		Applying Lemma \ref{lemma:orbit_maps_are_equivalent_to_uniform_convergence} we conclude that $B_{\A^{**}} \circledast \varphi$ is norm separable.
		
		The converse is a consequence of Lemma \ref{lemma:separable_implies_asplund}.
	\end{proof}
	
	\subsection{Tame Functionals}
	
	\begin{lemma} \label{lemma:adjoint_of_rosenthal_operator}
		Let $T\colon V \to W$ be a bounded operator such that $W$ is a Rosenthal space.
		Then $T^{*}(B_{W^{*}}) \subseteq V^{*}$ is co-tame.
	\end{lemma}
	\begin{proof}
		By contradiction, suppose that $T^{*}(B_{W^{*}})$ is not co-tame.
		We can therefore find a bounded $\{x_{n}\}_{n \in \N} \subseteq V$ which is independent over $T^{*}(B_{W^{*}})$.
		Note that we have:
		$$
		\langle T^{*}(\varphi), x_{n}\rangle = \langle \varphi, T(x_{n})\rangle
		$$
		for every $\varphi \in B_{W^{*}}, n \in \N$.
		As a consequence, $\{T(x_{n})\}$ is an independent sequence over $B_{W^{*}}$, a contradiction.
	\end{proof}
	
	\begin{thm}\label{thm:left_tame_right_cotame}
		A functional $\varphi \in \A^{*}$ is right tame if and only if $B_{\A^{**}} \circledast \varphi \underset{\ref{cor:dual_ball_action_is_weak_star_closure}}{=} \overline{B_{\A} \circledast \varphi}^{w^{*}}$ is co-tame.
		Similarly, $\varphi$ is left tame if and only if $\varphi \circledast B_{\A^{**}}$ is co-tame.
	\end{thm}
	\begin{proof}
		We only consider the right tame case, the other is virtually the same.
		First suppose that $\varphi \in \A^{*}$ is right tame.
		Applying Corollary \ref{cor:reflexive_asplund_rosenthal_factorization}, we can find a Rosenthal space $V$ and bounded maps $T\colon \A \to V$ and $S \colon V \to \A^{*}$ such that $\Qr_{\varphi}^{(r)} = S \circ T$ and $\Or_{\varphi}^{(l)} = T^{*} \circ S^{*}$.
		Applying Lemma \ref{lemma:adjoint_of_rosenthal_operator}, we conclude that $T^{*}(B_{V^{*}}) \subseteq \A^{*}$ is co-tame.
		However, 
		$$
		B_{\A^{**}} \circledast \varphi = 
		\Or_{\varphi}^{(l)}(B_{\A^{**}}) = 
		T^{*}\left( S^{*}(B_{\A^{**}})\right) \subseteq
		\lVert S^{*}\rVert T^{*}(B_{V^{*}}),
		$$
		and therefore $B_{\A^{**}} \circledast \varphi$ is indeed co-tame.
		
		Conversely, suppose that $B_{\A^{**}} \circledast \varphi$ is co-tame.
		By contradiction, let $\{\psi_{n}\}_{n \in \N} \subseteq \varphi \circledast B_{\A}$ be an independent sequence.
		By definition, we can find $\{a_{n}\}_{n \in \N} \subseteq B_{\A}$ such that $\psi_{n} = \varphi \circledast a_{n}$.
		For every $\mu \in B_{\A^{**}}$ we have:
		$$
		\langle \mu, \psi_{n} \rangle = 
		\langle \mu, \varphi \circledast a_{n}\rangle = 
		\langle \mu \circledast \varphi, a_{n}\rangle.
		$$
		This would imply that $\{a_{n}\}_{n \in \N}$ is an independent sequence over $B_{\A^{**}} \circledast \varphi$, a contradiction.
	\end{proof}
	
	\begin{f}\cite[Thm.~4.3]{meTameFunc}\label{fact:tame_implies_co_tame}
		Let $V$ be a Banach space.
		If $M \subseteq V^{*}$ is a weak-star compact, tame subset, then it is co-tame.
	\end{f}
	\begin{cor}
		Let $\varphi \in \A^{*}$.
		If $B_{\A^{**}} \circledast \varphi$ is \emph{tame}, then $\varphi$ is both left and right tame.
	\end{cor}
	\begin{proof}
		Clearly, $B_{\A} \circledast \varphi \subseteq B_{\A^{**}} \circledast \varphi$ so $\varphi$ is left tame.
		Using fact \ref{fact:tame_implies_co_tame} we can conclude that $B_{\A^{**}} \circledast \varphi$ is co-tame.
		Now, applying Theorem \ref{thm:left_tame_right_cotame} we get the desired result.
	\end{proof}
	
	\subsection{Summary}
	
	The following is a summary of Proposition \ref{prop:wap_equivalence}, Theorem \ref{thm:asplund_functionals}, \ref{thm:separable_asplund_iff_separable_arens_orbit} and Theorem \ref{thm:left_tame_right_cotame}.
	\begin{thm} \label{thm:summary}
		Let $\A$ be a Banach algebra and $\varphi \in \A^{*}$ be a functional.
		Then each column of the following table shows equivalent conditions, except for the cell marked with $\sharp$, which is equivalent when $\A$ is separable.		
		\begin{center}
			\begin{tabular}{||c || c | c | c ||} 
				\hline
				criterion & $\varphi \in \wap(\A)$ & $\varphi \in \asp_{r}(\A)$ & $\varphi \in \tame_{r}(\A)$ \\ 
				\hline\hline
				$\varphi \circledast B_{\A}$ is & 
				relatively weakly compact & 
				Asplund & 
				tame\\
				\hline
				$\Or_{\varphi}^{(l)}$ is & 
				(weak-star, weak) continuous & 
				\makecell{(weak-star, norm) fragmented \\ over $B_{\A^{**}}$ } & 
				-\\
				\hline
				$\left\{\arI_{\nu}^{(r)}\right\}_{\nu \in B_{\A^{**}}}$ is & \makecell{ a family of \\ (weak-star, $\rho_{\varphi}$) continuous functions } & \makecell{ (weak-star, $\rho_{\varphi}$) fragmented \\ family over $B_{\A^{**}}$} & 
				- \\
				\hline
				$B_{\A^{**}} \circledast \varphi$ is&
				weakly compact &
				separable $\sharp$&
				co-tame \\
				\hline
			\end{tabular}
		\end{center}
	\end{thm}
	Informally, this theorem implies that whenever $\varphi \circledast B_{\A}$ exhibits properties of the unit ball of some Banach space, then $B_{\A^{**}} \circledast \varphi$ exhibits the properties of the dual ball, and vice versa.
	
	\begin{q}
		Are there useful analogous characterizations of tameness of functionals using the orbit map or Arens products (the missing cells in the table)?
	\end{q}
	One notable idea currently missing in Theorem~\ref{thm:summary} is the characterization of weakly almost periodic functionals using the "Arens commutator" $\mu \arI \nu - \mu \arII \nu$.
	It has already been proposed as an interesting research direction \cite[Quest.~7.3]{meTameFunc}.
	We believe that the other connections in the table increase the suspicion that some analogy should be possible.
	
	\section{Further directions and examples}\label{section:directions_and_examples}
	This section is intended as a short collection of potential applications and future directions rather than a comprehensive treatment. The results below illustrate how the abstract criteria can be used in concrete settings; several natural problems and extensions are deferred to further work.
	
	\subsection{Operator Algebra}
	In this subsection we outline an application of the previous results to matrix coefficients of operator algebras.
	In \cite{Kaijser_1981} and later in \cite{Filali}, the authors show some natural circumstances in which weakly almost periodic functionals can be represented as matrix coefficients on reflexive representations.
	Here we take a different approach and present some sufficient conditions for such functionals to be Asplund/tame.
	
	Let $V$ be a Banach space and let $\A$ be a Banach sub-algebra of bounded operators together with composition.
	Clearly, $\A$ acts on $V$ from the left.
	It also acts from the right on $V^{*}$ via the map $\cdot\colon V^{*} \times \A\to V^{*}$ defined as:
	$$
	\forall x \in V, \phi \in V^{*}, a \in \A: \langle \phi \cdot a, x\rangle := \langle \phi, ax\rangle.
	$$
	For every $x \in V, \phi \in V^{*}$, consider the matrix coefficients $\Phi_{x, \phi}\in \A^{*}$ defined via:
	$$
	\forall b \in \A: \langle \Phi_{x, \phi}, b\rangle := \langle \phi, bx\rangle.
	$$
	
	\begin{lemma}\label{lemma:algebra_of_bounded_operators_on_matrix_coefficients}
		For every $a \in \A, x \in V, \phi \in V^{*}$ we have:
		$$
		a \circledast \Phi_{x, \phi} = \Phi_{ax, \phi}, 
		\Phi_{x, \phi} \circledast a = \Phi_{x, \phi \cdot a}.
		$$
	\end{lemma}
	\begin{proof}
		Indeed, for every $b \in \A$:
		\begin{align*}
			\langle a \circledast \Phi_{x, \phi}, b\rangle &=
			\langle \Phi_{x, \phi}, ba\rangle = 
			\langle \phi, ba x\rangle =
			\langle \Phi_{ax, \phi}, b\rangle,\\
			\langle \Phi_{x, \phi} \circledast a, b\rangle & =
			\langle \Phi_{x, \phi}, ab\rangle =
			\langle \phi, abx\rangle =
			\langle \phi \cdot a, bx\rangle =
			\langle \Phi_{x, \phi\cdot a}, b\rangle.
		\end{align*}
	\end{proof}
	For every $x \in V, \phi \in V^{*}$ we define $\Phi_{x, \cdot}\colon V^{*} \to \A^{*}, \Phi_{\cdot, \phi}\colon V \to \A^{*}$ as:
	$$
	\Phi_{x, \cdot}(\phi) := \Phi_{x, \phi},\ \Phi_{\cdot, \phi}(x) := \Phi_{x, \phi}.
	$$
	Also, we will write $\widehat{\Or}_{\phi}^{(l)}\colon V^{**} \to \A^{*}$, $\widehat{\Or}_{x}^{(r)}\colon V^{*} \to \A^{*}$ via:
	\begin{align*}
		\left\langle \widehat{\Or}_{\phi}^{(l)}(\omega), b \right\rangle := \langle \omega, \phi \cdot b \rangle,\ 
		\left\langle \widehat{\Or}_{x}^{(r)}(\theta), b \right\rangle := \langle \theta, bx \rangle,
	\end{align*}
	for every $b \in \A$, $\theta \in V^{*}$ and $\omega \in V^{**}$.
	
	In the following passages we will abbreviate $\Or_{\Phi_{x, \phi}}^{(l)}$ and $\Or_{\Phi_{x, \phi}}^{(r)}$ to simply $\Or_{x, \phi}^{(l)}$ and $\Or_{x, \phi}^{(r)}$.
	\begin{lemma} \label{lemma:operator_orbit_factor}
		Let $x \in V, \phi \in V^{*}$. Then:
		\begin{align*}
			\Or_{x, \phi}^{(l)} &= \widehat{\Or}_{\phi}^{(l)} \circ \Phi_{x, \cdot}^{*},\\
			\Or_{x, \phi}^{(r)} &= \widehat{\Or}_{x}^{(r)} \circ \Phi_{\cdot, \phi}^{*}.
		\end{align*}
	\end{lemma}
	\begin{proof}
		For every $b \in \A, \mu \in \A^{**}$:
		\begin{align*}
			\langle \Or_{x, \phi}^{(l)}(\mu), b\rangle & = 
			\langle \mu \circledast \Phi_{x, \phi}, b\rangle\\
			& = \langle \mu, \Phi_{x, \phi} \circledast b\rangle\\
			& \underset{\ref{lemma:algebra_of_bounded_operators_on_matrix_coefficients}}{=}
			\langle \mu, \Phi_{x, \phi \cdot b}\rangle\\
			& = \langle \mu, \Phi_{x, \cdot}(\phi \cdot b)\rangle\\
			& = \langle \Phi_{x, \cdot}^{*}(\mu), \phi \cdot b\rangle\\
			& = \left\langle \widehat{\Or}_{\phi}^{(l)}\left(\Phi_{x, \cdot}^{*}(\mu) \right), b \right\rangle,\\
			\langle \Or_{x, \phi}^{(r)}(\mu), b\rangle & = 
			\langle \Phi_{x, \phi} \circledast \mu, b\rangle\\
			& = \langle \mu, b\circledast \Phi_{x, \phi}\rangle \\
			& \underset{\ref{lemma:algebra_of_bounded_operators_on_matrix_coefficients}}{=}
			\langle \mu, \Phi_{bx, \phi}\rangle\\
			& = \langle \mu, \Phi_{\cdot, \phi}(bx)\rangle\\
			& = \langle \Phi_{\cdot, \phi}^{*}(\mu), bx\rangle\\
			& = \left\langle \widehat{\Or}_{x}^{(r)}(\Phi_{\cdot, \phi}^{*}(\mu)), b\right\rangle.
		\end{align*}
	\end{proof}
	
	Thus, we get the following sufficient conditions.
	\begin{prop} \label{prop:sufficient_conditions_for_matrix_coefficients}
		Let $x \in V, \phi \in V^{*}$.
		Then:
		\ben
		\item If $\widehat{\Or}_{\phi}^{(l)}\colon B_{V^{**}} \to \A^{*}$ is (weak-star, norm) fragmented, then $\Phi_{x, \phi}$ is right Asplund.
		\item If $\widehat{\Or}_{x}^{(r)}\colon B_{V^{*}} \to \A^{*}$ is (weak-star, norm) fragmented, then $\Phi_{x, \phi}$ is left Asplund.
		\item If $\widehat{\Or}_{\phi}^{(l)}(B_{V^{**}})$ is separable (resp. co-tame), then $\Phi_{x, \phi}$ is right Asplund (resp. tame).
		\item If $\widehat{\Or}_{x}^{(r)}(B_{V^{*}})$ is separable (resp. co-tame), then $\Phi_{x, \phi}$ is left Asplund (resp. tame).
		\een
	\end{prop}
	\begin{proof}~
		
		\ben
		\item In virtue of \cite[Lemma~2.3.2]{GM-rose}, the composition $f \circ \alpha$ of a fragmented map $f\colon Y \to Z$ with a continuous map $\alpha\colon X \to Y$ is fragmented.
		Also, applying Fact \ref{fact:adjoint_is_weak_and_weak_star_cont}, $\Phi^{*}_{x, \cdot}$ is (weak-star, weak-star) continuous, and $\Phi^{*}_{x, \cdot}(B_{\A^{**}}) \subseteq r B_{V^{**}}$ for some $r > 0$.
		Using Lemma \ref{lemma:operator_orbit_factor} we conclude that if $\widehat{\Or}_{\phi}^{(l)}\colon B_{V^{**}} \to \A^{*}$ is (weak-star, norm) fragmented then so is $\widehat{\Or}_{\phi}^{(l)}\colon rB_{V^{**}} \to \A^{*}$ and thus the composition $\Or_{x, \phi}^{(l)} = \widehat{\Or}_{\phi}^{(l)} \circ \Phi_{x, \cdot}^{*}$ is also (weak-star, norm) fragmented.
		By Theorem \ref{thm:asplund_functionals}, this implies that $\Phi_{x, \phi}$ is right Asplund.
		\item Analogous.
		\item Adjoint maps are bounded, so there exists some $r > 0$ such that $\Phi_{x, \cdot}^{*}(B_{\A^{**}}) \subseteq r B_{V^{**}}$.
		Also, in virtue of Lemma \ref{lemma:operator_orbit_factor} we know that $\Or_{x, \phi}^{(l)} = \widehat{\Or}_{\phi}^{(l)} \circ \Phi_{x, \cdot}^{*}$.
		Therefore:
		$$
		\Or_{x, \phi}^{(l)}(B_{\A^{**}}) = 
		\widehat{\Or}_{\phi}^{(l)} \left( \Phi_{x, \cdot}^{*}(B_{\A^{**}})\right) \subseteq r \widehat{\Or}_{\phi}^{(l)}(B_{V^{**}}).
		$$
		As a consequence, $\widehat{\Or}_{\phi}^{(l)}(B_{V^{**}})$ being separable (resp. co-tame) implies that $\Or_{x, \phi}^{(l)}(B_{\A^{**}})$ is separable (resp. co-tame).
		Now, by Lemma \ref{lemma:separable_implies_asplund} (resp. Theorem \ref{thm:left_tame_right_cotame}), this is equivalent to $\Phi_{x, \phi}$ being right Asplund (resp. tame).
		\item Analogous.
		\een
	\end{proof}
	
	\begin{cor}
		If $V^{*}$ is separable, and $\A \subseteq L(V)$ is some Banach algebra, then every matrix coefficient $\Phi_{x, \phi}$ for $x \in V, \phi \in V^{*}$ is left Asplund.
	\end{cor}
	\begin{ex}
		Let $V = c_{0}$ be the space of sequences converging to $0$ and $\A := L(V)$ be the Banach algebra of bounded operators of $V$.
		Then in virtue of the previous corollary, for every $x \in V, \phi \in V^{*}$ we get that $\Phi_{x, \phi}$ is left Asplund.
	\end{ex}
	
	\subsection{Group Representation}
	In this subsection we will apply the previous proposition to the case of topological group representation on Banach spaces, and their relationship with matrix coefficient.
	It is known that every weakly almost periodic, Asplund or tame function on a topological group can be represented as matrix coefficients coming from a representation on reflexive, Asplund or Rosenthal spaces respectively \cite{Me-Frag04, GM-rose}.
	We will focus on a different perspective, showing some sufficient conditions for a function coming from some strongly continuous Banach representation to be Asplund/tame.
	
	Let $G$ be a topological group.
	If $f \in \R^{G}$ and $g \in G$, we consider the right translation $f \cdot g \in \R^{G}$ defined via:
	$$
	\forall h \in G: (f \cdot g)(h) := f(gh).
	$$
	
	\begin{defin}
		A \emph{strongly continuous representation} of $G$ on the Banach space $V$ is a strongly continuous co-homomorphism $\alpha\colon G \to \Iso(V)$.
		In this case we will write $x \cdot g =(\alpha(g))(x)$.
	\end{defin}
	
	Let $\alpha\colon G \to \Iso(V)$ be a strongly continuous action.
	Write $\A := \overline{\spn}\ \alpha(G) \leq L(V)$.
	For every $x \in V, \phi \in V^{*}$ the matrix coefficients $\Phi_{x, \phi} \in \A^{*}$ induce a continuous function $f_{x, \phi} \in C_{b}(G)$ defined via:
	$$
	f_{x, \phi}(g) := \langle \Phi_{x, \phi}, \alpha(g)\rangle = \langle \phi, (\alpha(g))(x)\rangle = \langle \phi, x\cdot g\rangle.
	$$
	\begin{lemma}\label{lemma:matrix_coefficients_on_representation}
		For every $x \in V$, $\phi \in V^{*}$ and $g \in G$ we have:
		$$
		f_{x, \phi} \cdot g = f_{x\cdot g, \phi}.
		$$
	\end{lemma}
	\begin{proof}
		Indeed, for every $h \in G$:
		\begin{align*}
			(f_{x, \phi} \cdot g)(h) & =
			f_{x, \phi}(gh) =
			\langle \phi, x \cdot (gh)\rangle =
			\langle \phi, (x \cdot g) \cdot h\rangle =
			f_{x\cdot g, \phi}(h).
		\end{align*}
	\end{proof}
	Write $\RUC(G)$ for the space of bounded, right uniformly continuous functions on $G$.
	\begin{lemma}
		For every $x \in V$, $\phi\in V^{*}$ we have $f_{x, \phi} \in \RUC(G)$.
	\end{lemma}
	\begin{proof}
		Let $\eps > 0$.
		If $\phi = 0$, then $f_{x, \phi} \equiv 0$ and there is nothing to prove.
		We need to find some neighborhood $e \in U \subseteq G$ such that:
		$$
		\forall g, h \in G: \left( gh^{-1} \in U \Rightarrow \lVert f_{x, \phi} \cdot g - f_{x, \phi} \cdot h \rVert < \eps \right).
		$$
		Since $\alpha$ is strongly continuous, we can find some neighborhood $e \in U \subseteq G$ such that:
		$$
		\forall g \in U: \lVert (\alpha(g) - id)(x)\rVert = \lVert x \cdot g - x\rVert < \frac{1}{\lVert \phi \rVert} \eps.
		$$
		Now, if $gh^{-1} \in U$ then:
		\begin{align*}
			\lVert f_{x, \phi} \cdot g - f_{x, \phi} \cdot h\rVert & 
			\underset{\ref{lemma:matrix_coefficients_on_representation}}{=}
			\lVert f_{x \cdot g, \phi} - f_{x \cdot h, \phi} \rVert\\
			& = \sup\limits_{s \in G} \lvert f_{x \cdot g, \phi}(s) - f_{x \cdot h, \phi}(s)\rvert\\
			& = \sup\limits_{s \in G} \lvert \langle \phi, (x \cdot g) \cdot s\rangle - \langle \phi, (x \cdot h) \cdot s\rangle\rvert\\
			& = \sup\limits_{s \in G} \lvert \langle \phi, x \cdot g s - x \cdot h s\rangle\rvert\\
			& \leq \sup\limits_{s \in G} \lVert \phi\rVert \lVert x \cdot g s - x \cdot h s\rVert\\
			& \underset{\alpha(G)\subseteq \Iso(V)}{=} 
			\sup\limits_{s \in G} \lVert \phi\rVert \lVert x \cdot g s (s^{-1}h^{-1}) - x \cdot h s (s^{-1}h^{-1})\rVert\\
			& = \sup\limits_{s \in G} \lVert \phi\rVert \lVert x \cdot g h^{-1} - x\rVert\\
			& \leq \lVert \phi\rVert \frac{1}{\lVert \phi \rVert}\eps = \eps.
		\end{align*}
	\end{proof}
	
	It is said that a function $f \in \RUC(G)$ is right weakly almost periodic/Asplund/tame if the orbit $f \cdot G$ is relatively weakly compact/Asplund/tame respectively.
	\begin{prop} \label{prop:sufficient_conditions_for_group_matrix_coefficients}
		If the functional $\Phi_{x, \phi} \in \A^{*}$ is left Asplund (resp. left tame) in $\A^{*}$, then the function $f_{x, \phi} \colon G \to \R$ is right Asplund (resp. right tame) in $\RUC(G)$.
	\end{prop}
	\begin{proof}
		Consider the restriction map $R\colon \A^{*} \to \ell^{\infty}(G)$, where $\ell^{\infty}(G)$ is the Banach space of bounded functions on $G$ with supremum norm, defined via:
		$$
		\forall \varphi \in \A^{*}, g \in G: (R(\varphi))(g) := \langle \varphi, \alpha(g)\rangle.
		$$
		Note that:
		$$
		\forall \varphi \in \A^{*}: \lVert R(\varphi)\rVert =
		\sup\limits_{g \in G} \lvert (R(\varphi))(g)\rvert =
		\sup\limits_{g \in G} \lvert \langle \varphi, \alpha(g)\rangle\rvert \leq
		\sup\limits_{g \in G} \lVert\varphi\rVert \lVert\alpha(g)\rVert \leq
		\lVert \varphi\rVert,
		$$
		where the last inequality is a consequence of $\alpha(G) \subseteq \Iso(V) \subseteq B_{\A}$.
		In other words, $R$ is a bounded map.
		By definition, $\Phi_{x, \phi}$ is left Asplund (resp. tame) if and only if $B_{\A} \circledast \Phi_{x, \phi}$ is Asplund (resp. tame).
		This property is preserved under bounded linear maps (Fact \ref{fact:image_of_small_subset}), so $R(B_{\A} \circledast \Phi_{x, \phi})$ is also Asplund (resp. tame). 
		Moreover, we claim that $f_{x, \phi} \cdot G \subseteq R(B_{\A} \circledast \Phi_{x, \phi})$, and is therefore also Asplund (resp. tame).
		Indeed, for every $g \in G$:
		\begin{align*}
			f_{x, \phi} \cdot g & \underset{\ref{lemma:matrix_coefficients_on_representation}}{=}
			f_{x \cdot g, \phi}\\
			& = R(\Phi_{x \cdot g, \phi})\\
			& = R(\Phi_{(\alpha(g))x, \phi})\\
			& \underset{\ref{lemma:algebra_of_bounded_operators_on_matrix_coefficients}}{=}
			R(\alpha(g) \circledast \Phi_{x, \phi})\\
			& \in R(B_{\A} \circledast \Phi_{x, \phi}).
		\end{align*}
	\end{proof}
	
	Together with Proposition \ref{prop:sufficient_conditions_for_matrix_coefficients}, we get the following.
	\begin{prop} \label{prop:sufficient_conditions_for_group_function}
		Suppose that $\alpha\colon G \to \Iso(V)$ is a strongly continuous representation of the topological group $G$ on the Banach space $V$.
		Let $\A$ be the Banach algebra spanned by $\alpha(G)$ in $L(V)$ and $x \in V, \phi \in V^{*}$.
		Then:
		\ben
		\item If $\widehat{\Or}_{x}^{(r)}\colon B_{V^{*}} \to \A^{*}$ is (weak-star, norm) fragmented, then $f_{x, \phi} \in \RUC(G)$ is right Asplund.
		\item If $\widehat{\Or}_{x}^{(r)}(B_{V^{*}})$ is separable (resp. co-tame), then $f_{x, \phi}$ is right Asplund (resp. tame).
		\een
	\end{prop}
	\begin{remark}
		Unlike most results in this paper, some extra care needs to be taken in order to obtain analogous left-Asplund and left-tame statements for \ref{prop:sufficient_conditions_for_group_matrix_coefficients} and \ref{prop:sufficient_conditions_for_group_function}.
		Specifically, we would need to consider a left representation $\alpha\colon G \to \Iso(V)$ which is a homomorphism and not co-homomorphism.
	\end{remark}
	
	\begin{q}
		In which cases, if any, do the sufficient conditions of Proposition \ref{prop:sufficient_conditions_for_group_function} become necessary?
	\end{q}
	
	\section{The Banach Algebra $L^{1}(G)$ of a Locally Compact Group}\label{section:group_algebra}
	The goal of this section is to give a concrete description of right Asplund and tame elements of $L^{\infty}(G)$ in terms of orbits under finitely additive $\{0, 1\}$-valued measures. The strategy is to combine Theorem \ref{thm:separable_asplund_iff_separable_arens_orbit} and \ref{thm:left_tame_right_cotame} with the structural description of the bidual via the Yosida–Hewitt theorem.
	
	Let $G$ be a locally compact group and let $\lambda$ be its left Haar measure.
	We will assume that $G$ is $\sigma$-finite and that $\lambda$ is complete.
	Write $\Sigma$ for the appropriate $\sigma$-algebra.
	Throughout this section we will consider the Banach algebra $\A := L^{1}(G)$ whose multiplication is convolution, defined via:
	$$
	\forall a, b \in L^{1}(G), g \in G: (a \star b)(g) := \int\limits_{G} a(h) b(h^{-1}g) d\lambda(h).
	$$
	It is then known that $\A^{*}$ can be identified with $L^{\infty}(G)$ where:
	$$
	\forall a \in L^{1}(G), \varphi \in L^{\infty}(G): \langle \varphi, a \rangle = \int\limits_{G} a(g) \varphi(g) d\lambda(g).
	$$
	
	\begin{remark}
		The results of this section are formulated for right Asplund or tame functionals.
		The derivation of left analogues is possible using exactly the same logic.
	\end{remark}
	
	\begin{f}[Yosida–Hewitt theorem]\cite[Thm.~3.1]{Toland}
		The dual $\A^{**} = \left( L^{\infty}(G)\right)^{*}$ can be isometrically identified with the space $\ba_{\lambda}(G)$ of finitely additive measures on $G$ whose family of null sets contain the null sets of $\lambda$.
	\end{f}
	
	\begin{defin}\cite[Section~5]{Toland}
		Write 
		$$
		\mathfrak{G} := \{\omega \in \ba_{\lambda}(G)\mid \omega(G) = 1 \text{ and } \forall A \in \Sigma: \omega(A) \in \{0, 1\} \}.
		$$
	\end{defin}
	In virtue of the Yosida-Hewitt theorem, we can consider $\mathfrak{G}$ as a subset of $\A^{**} = \left(L^{\infty}(G)\right)^{*}$.
	Therefore, we can furnish it with the weak-star topology of $\A^{**}$.
	\begin{f}\label{fact:properties_of_G_compactification}~\ben
		\item $\mathfrak{G}$ is a weak-star closed (hence compact) subset of $\A^{**}$ \cite[Lemma~7.13]{Toland}.
		\item $L^{\infty}(G)$ and $C(\mathfrak{G})$ are isometrically isomorphic \cite[Thm.~7.4 and Lemma~7.12]{Toland}.
		\item $\mathfrak{G} \cup \left(-\mathfrak{G}\right) = \ext(B_{\ba_{\lambda}(G)})$ \cite[Lemma~7.14]{Toland}.
		\een
	\end{f}
	Combining the last property with the Krein–Milman theorem \cite[p.~II.55, Thm.~1]{bourbTVS}, we get the following corollary.
	\begin{cor} \label{cor:compactification_as_extreme_points}
		$$
		\overline{\acx}^{w^{*}} \mathfrak{G} = B_{\A^{**}}.
		$$
	\end{cor}
	
	\begin{f}\cite[Section~5.1]{Toland}\label{fact:stone_cech_compactification}
		If $G$ is discrete, then $\mathfrak{G}$ is the same as the Stone-\v{C}ech compactification $\beta G$ of $G$.
	\end{f}

	\begin{f}\cite[p.~223]{Folland} \label{fact:second_countable_groups}
		If $G$ is a second-countable locally compact group, then $L^{1}(G)$ is separable.
	\end{f}

	\begin{lemma} \label{lemma:characteristic_orbit}
		If $G$ is discrete, then for every $X \subseteq G$ we have:
		$$
		  \mathfrak{G} \circledast 1_{X} = 1_{\mathfrak{G} \cdot X},
		$$
		where
		$$
		  \forall \omega \in \mathfrak{G}: \omega \cdot X := \{g \in G \mid \omega \left( g^{-1} \cdot X\right) = 1 \}.
		$$ 
	\end{lemma}
	\begin{proof}
		For every $g \in G$, write $e_{g} \in L^{1}(G)=\ell^{1}(G)$ for the standard basis element with support at $g$.
		First, note that for every $g, h \in G$ we have:
		\begin{align*}
			\langle 1_{X} \circledast e_{g}, e_{h}\rangle = 
			\langle 1_{X}, e_{g} \circledast e_{h}\rangle = 
			\langle 1_{X}, e_{gh}\rangle = \begin{cases}
				1 & gh \in X\\
				0 & \text{otherwise}
			\end{cases}
		\end{align*}
		In other words, $1_{X} \circledast e_{g} = 1_{g^{-1} \cdot X}$.
		Now, for every $\omega \in \mathfrak{G}$:
		$$
		  \langle \omega \circledast 1_{X}, e_{g}\rangle =
		  \langle \omega, 1_{X} \circledast e_{g}\rangle =
		  \langle \omega, 1_{g^{-1} \cdot X}\rangle = \omega (1_{g^{-1} \cdot X}).
		$$
	\end{proof}
	
	\begin{thm} \label{thm:asplund_subsets_on_locally_compact_groups}
		Let $G$ be a locally compact group such that $L^{1}(G)$ is separable (e.g., second-countable groups).
		Then a functional $\varphi \in L^{\infty}(G)$ is right Asplund if and only if $\mathfrak{G} \circledast \varphi$ is separable.
		
	\end{thm}
	\begin{proof}
		Because $L^{1}(G)$ is separable, Theorem \ref{thm:separable_asplund_iff_separable_arens_orbit} implies that $\varphi$ is right Asplund if and only if $B_{\A^{**}} \circledast \varphi$ is separable.
		In particular, $\varphi$ is right Asplund implies that $\mathfrak{G} \circledast \varphi$ is separable
		
		Conversely, suppose that $\mathfrak{G} \circledast \varphi$ is indeed separable. 
		We will show that $B_{\A^{**}}\circledast \varphi$ is also separable.
		Recall that $\mathfrak{G}$ is weak-star compact (Fact \ref{fact:properties_of_G_compactification}).
		Applying Lemma \ref{lemma:bidual_action_is_orbit_weak_star_continuous} we conclude that $\mathfrak{G} \circledast \varphi$ is also weak-star compact.
		By Lemma \ref{lemma:separable_implies_cotame}, $\mathfrak{G} \circledast \varphi$ is separable and thus co-tame.
		Now, using Fact \ref{fact:co_tame_milman} we conclude that:
		$$
		\overline{\co}^{w^{*}} (\mathfrak{G} \circledast \varphi) = \overline{\co} (\mathfrak{G} \circledast \varphi).
		$$
		Similarly:
		$$
		\overline{\acx}^{w^{*}} (\mathfrak{G} \circledast \varphi) = \overline{\acx} (\mathfrak{G} \circledast \varphi).
		$$
		Moreover, we have:
		$$
		B_{\A^{**}} \circledast \varphi =
		\Or_{\varphi}^{(l)}(B_{\A^{**}}) \underset{\ref{cor:compactification_as_extreme_points}}{=}
		\Or_{\varphi}^{(l)}(\overline{\acx}^{w^{*}} \mathfrak{G}) \underset{\ref{lemma:bidual_action_is_orbit_weak_star_continuous}}{\subseteq}
		\overline{\acx}^{w^{*}} \Or_{\varphi}^{(l)}(\mathfrak{G}) =
		\overline{\acx}^{w^{*}} (\mathfrak{G}\circledast \varphi) = 
		\overline{\acx}(\mathfrak{G} \circledast \varphi).
		$$
		In other words, $B_{\A^{**}} \circledast \varphi$ is contained in the absolutely closed convex hull of a separable space, and thus separable.
	\end{proof}

	WAP sets were defined in \cite{WAPRup} as subsets of semitopological semigroups whose characteristic function is weakly almost periodic in a dynamical sense.
	In \cite{GM-MTame}, the authors study and characterize \emph{Asplund and tame} subsets of $\Z$, which are defined analogously.
	Using the previous theorem, we can prove the following generalization of their results.
	\begin{prop} \label{prop:asplund_subsets_in_discrete_groups}
		Let $G$ be a discrete, countable group.
		Then a subset $X \subseteq G$ is right Asplund (i.e., $1_{X} \in L^{\infty}(G)$ is right Asplund) if and only if $\mathfrak{G} \circledast 1_{X}$ is countable.
	\end{prop}
	\begin{proof}
		Clearly $G$ is second countable so Theorem \ref{thm:asplund_subsets_on_locally_compact_groups} is applicable.
		Therefore $1_{X}$ is right Asplund if and only if $\mathfrak{G} \circledast 1_{X}$ is separable.
		By Lemma \ref{lemma:characteristic_orbit}, $\mathfrak{G} \circledast 1_{X}$ consists of characteristic functions. 
		As a consequence, it is discrete, hence separable if and only if it is countable.
	\end{proof}
	\begin{remark}
		In \cite{WAPRup} and \cite{GM-MTame}, they consider a \emph{function} to be WAP/Asplund/tame if its orbit under translations is weakly relatively compact/Asplund/tame.
		We consider $1_{X}$ as a \emph{functional} in $L^{\infty}(G) = \ell^{\infty}(G)$.
		However, both of these notions are equivalent as was shown in \cite[Thm.~3.2]{TameFunc} for discrete groups and later in \cite[Thm.~6.3]{meTameFunc} for right uniformly continuous functionals on locally compact groups.
	\end{remark}
    
	The last proposition leads to a new proof for the characterization of Asplund subsets of $\Z$.
	\begin{f}\cite[Thm.~4.11]{GM-MTame} \label{fact:asplund_subsets_equivalence_new}
		Let $X \subseteq \Z$ be some subset.
		The following are equivalent:
		\ben
		\item $X$ is an Asplund subset.
		\item there exists a countable $Y\subseteq \beta \Z$ such that for every $\omega \in \beta\Z$ we can find $\theta \in Y$ such that:
		$$
		\forall n \in \Z: n + X \in \omega \iff n + X \in \theta.
		$$
		\een
	\end{f}
	\begin{proof}
		First, by Fact \ref{fact:stone_cech_compactification} we have $\beta\Z = \mathfrak{G}$.
		In virtue of Proposition \ref{prop:asplund_subsets_in_discrete_groups}, $X$ is an Asplund subset if and only if $\beta \Z \circledast 1_{X}$ is countable.
		
		We claim that $\omega \circledast 1_{X} = \theta \circledast 1_{X}$ for $\omega, \theta \in \beta \Z$ if and only if:
		$$
		\forall n \in \Z: n + X \in \omega \iff n + X \in \theta.
		$$
		First, a straight-forward calculation yields that for every element $e_{n}$ of the ``standard basis" of $L^{\infty}(\Z)$, we have:
		$$
		1_{X} \circledast 1_{n} = 1_{X - n}.
		$$
		Indeed:
		\begin{align*}
			\omega \circledast 1_{X} = \theta \circledast 1_{X} & \iff
			\forall n \in \Z: \langle \omega \circledast 1_{X}, e_{n}\rangle = \langle \theta \circledast 1_{X}, e_{n}\rangle\\
			& \iff \forall n\in \Z: \langle \omega, 1_{X} \circledast e_{n}\rangle = \langle \theta, 1_{X} \circledast e_{n}\rangle\\
			& \iff \forall n\in \Z: \langle \omega, 1_{X-n}\rangle = \langle \theta, 1_{X-n}\rangle\\
			& \iff \forall n\in \Z: X - n \in \omega \iff X - n \in \theta\\
			& \iff \forall n\in \Z: X + n \in \omega \iff X + n \in \theta.
		\end{align*}
		
		Now, it is easy to see that the latter condition of our statement is satisfied if and only if there exists a countable $Y \subseteq \beta \Z$ such that $\beta \Z \circledast 1_{X} = Y \circledast 1_{X}$.
		This is clearly equivalent to $\beta\Z \circledast 1_{X}$ being countable.
	\end{proof}
	
	\begin{thm} \label{thm:right_tame_left_co_tame_locally_compact}
		Let $G$ be a locally compact group. 
		Then a functional $\varphi \in L^{\infty}(G)$ is right tame if and only if $\mathfrak{G} \circledast \varphi$ is co-tame.
	\end{thm}
	\begin{proof}
		By Theorem \ref{thm:left_tame_right_cotame}, $\varphi$ is right tame if and only if $B_{\A^{**}} \circledast \varphi$ is co-tame.
		By Corollary \ref{cor:compactification_as_extreme_points} we have:
		$$
		B_{\A^{**}} = \overline{\acx}^{w^{*}} \mathfrak{G}.
		$$
		Therefore:
		$$
		B_{\A^{**}} \circledast \varphi = \Or_{\varphi}^{(l)}(B_{\A^{**}})
		= \Or_{\varphi}^{(l)}\left(\overline{\acx}^{w^{*}} \mathfrak{G}\right) 
		\underset{\ref{lemma:bidual_action_is_orbit_weak_star_continuous}}{\subseteq} \overline{\acx}^{w^{*}} \Or_{\varphi}^{(l)} (\mathfrak{G}) =
		\overline{\acx}^{w^{*}} \left( \mathfrak{G} \circledast \varphi\right).
		$$
		Co-tame subsets are closed under convex hulls and weak-star closures \cite[Lemma~3.30]{TameFunc}, so $\varphi$ is right tame if and only if $\mathfrak{G} \circledast \varphi$ is co-tame.
	\end{proof}
	
	\begin{q}
		In \cite{GM-MTame}, the authors also prove a similar version of Fact \ref{fact:asplund_subsets_equivalence_new} for tame characteristic functions over $\Z$.
		Is there a useful tame analogue of Proposition \ref{prop:asplund_subsets_in_discrete_groups}?
	\end{q}


	\bibliography{mybib}{}
	\bibliographystyle{plain}
\end{document}